\documentclass[11pt,a4paper]{article}
\usepackage{amsfonts,amsmath,amsthm}
\usepackage{amstext}
\usepackage{graphicx}
\usepackage[pdftex]{hyperref} 
\usepackage{subfig}
\usepackage{verbatim}
\graphicspath{{figures/}}
\newtheorem{theorem}{Theorem}[section]
\newtheorem{lemma}[theorem]{Lemma}

\newtheorem{corol}[theorem]{Corollary}

\newtheorem{erratum}[theorem]{Erratum}
\def\minim{\mathop{\hbox{minimize}}}

\def\minimize#1{\displaystyle\minim_{#1}}

\newenvironment{remark}%
  {\par\medbreak\refstepcounter{theorem}%
    \noindent\textbf{Remark~\thetheorem. }}%
  {\par\medskip}

\newcommand{\vz}[1]{\ensuremath{\mathbb{#1}}}

\newcommand{\R}{{\vz R}}

\newcommand{\dvg}{\text{div}\,}

\DeclareMathOperator{\supp}{supp}
\def\pref#1{(\ref{#1})}

\long\def\drop#1{}

\let\e\varepsilon
\let\epsilon\varepsilon

\def\XXint#1#2#3{{\setbox0=\hbox{$#1{#2#3}{\int}$}
     \vcenter{\hbox{$#2#3$}}\kern-.5\wd0}}

\DeclareMathAlphabet{\mathpzc}{OT1}{pzc}{m}{it}

\usepackage{geometry}                
\usepackage{graphicx}
\usepackage{amssymb}
\usepackage{epstopdf}
\usepackage{latexsym}
\usepackage{amssymb}
\usepackage{amsfonts}
\usepackage{bm}
\usepackage{epsfig}
\usepackage{bbm}
\usepackage{amsmath}

\usepackage{psfrag}
\usepackage{color}
\usepackage{multicol}
\usepackage{subfig}
\usepackage{appendix}

\title{Anisotropic Total Variation Regularized $L^1$-Approximation and Denoising/Deblurring of 2D Bar Codes
}
\author{Rustum Choksi\footnote{Department of Mathematics and Statistics, McGill University, Burnside Hall, Room 1005, 805 Sherbrooke Street West, Montreal, Quebec, H3A 2K6, Canada}\qquad  Yves van Gennip\footnote{Department of Mathematics, University of California Los Angeles, 520 Portola Plaza, Math Sciences Building 6363, Los Angeles, California, 90095, USA}\qquad  Adam Oberman\footnote{Department of Mathematics, Simon Fraser University, 
Burnaby, British Columbia, 8888 University Drive, V5A 1S6, Canada}}

\begin{document}

\maketitle

\begin{abstract}
We consider variations of the Rudin-Osher-Fatemi functional which are particularly well-suited to denoising and deblurring of 2D bar codes. These functionals consist of an anisotropic total variation favoring rectangles and a fidelity term which measure  the $L^1$ distance to the signal, both with and without the presence of  a deconvolution operator.  
Based upon the existence of a certain associated vector field, 
we find necessary and sufficient conditions for a function to be a  minimizer. We apply these results to 2D bar codes to find explicit regimes --  in terms of  the fidelity parameter and smallest length scale of the bar codes -- for which the perfect bar code is attained  via minimization of the functionals. 
Via a discretization   reformulated as a linear program, we perform numerical experiments for all functionals demonstrating their denoising and deblurring capabilities. 



\medskip
\end{abstract}

\section{Introduction}
In this article we study the application of total variation-based energy minimization for denoising and deblurring of 2D bar codes. 
A 2D bar code is a collection of non-overlapping black squares, the lengths of whose sides are all bounded below by some value $\omega$, placed on a white backdrop. Analogous to the terminology for 1D bar codes, we call the lower bound $\omega$ the $X$-dimension of the bar code.  
Examples include stacked and matrix 2D bar codes illustrated in Figure 1 
(see \cite{Palmer07} for a thorough description of  2D bar code symbologies).  
When these bar codes are scanned,  the resulting signal will be a blurred and noisy version of the original bar code. Efficient and robust techniques to recover the original bar code are needed to retrieve the information from the code.

The problem of denoising and deblurring images via variational methods has received much attention in the literature since the introduction of the Rudin-Osher-Fatemi (ROF) functional in \cite{RudinOsherFatemi92-2}. This functional is the sum of a so called fidelity term, which measures the $L^2$ distance between the argument of the functional $u$ and the given (measured) signal $f$, and the total variation of $u$, which acts as a regularization term. In this paper we will study variations of this functional that take into account the \emph{a priori} knowledge that the original image which we want to recover is a 2D bar code.

\begin{figure} 
\centerline{ {\includegraphics[viewport = 20 30 430 420, clip, width=1.7in]{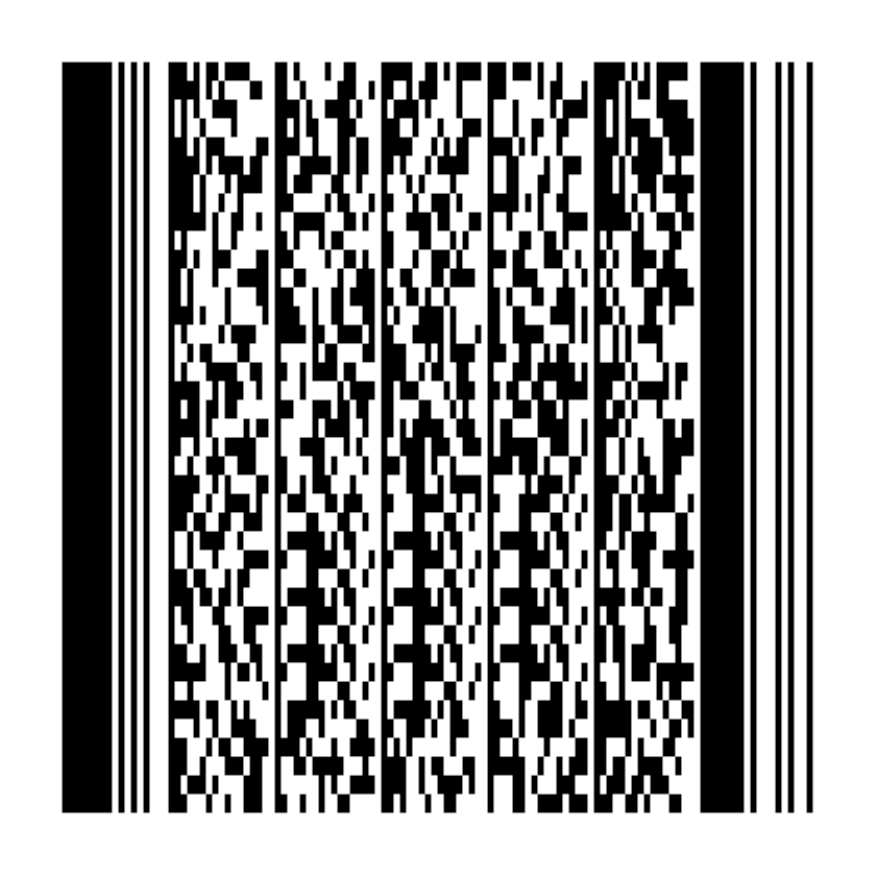}} \qquad {\includegraphics[width=1.5in]{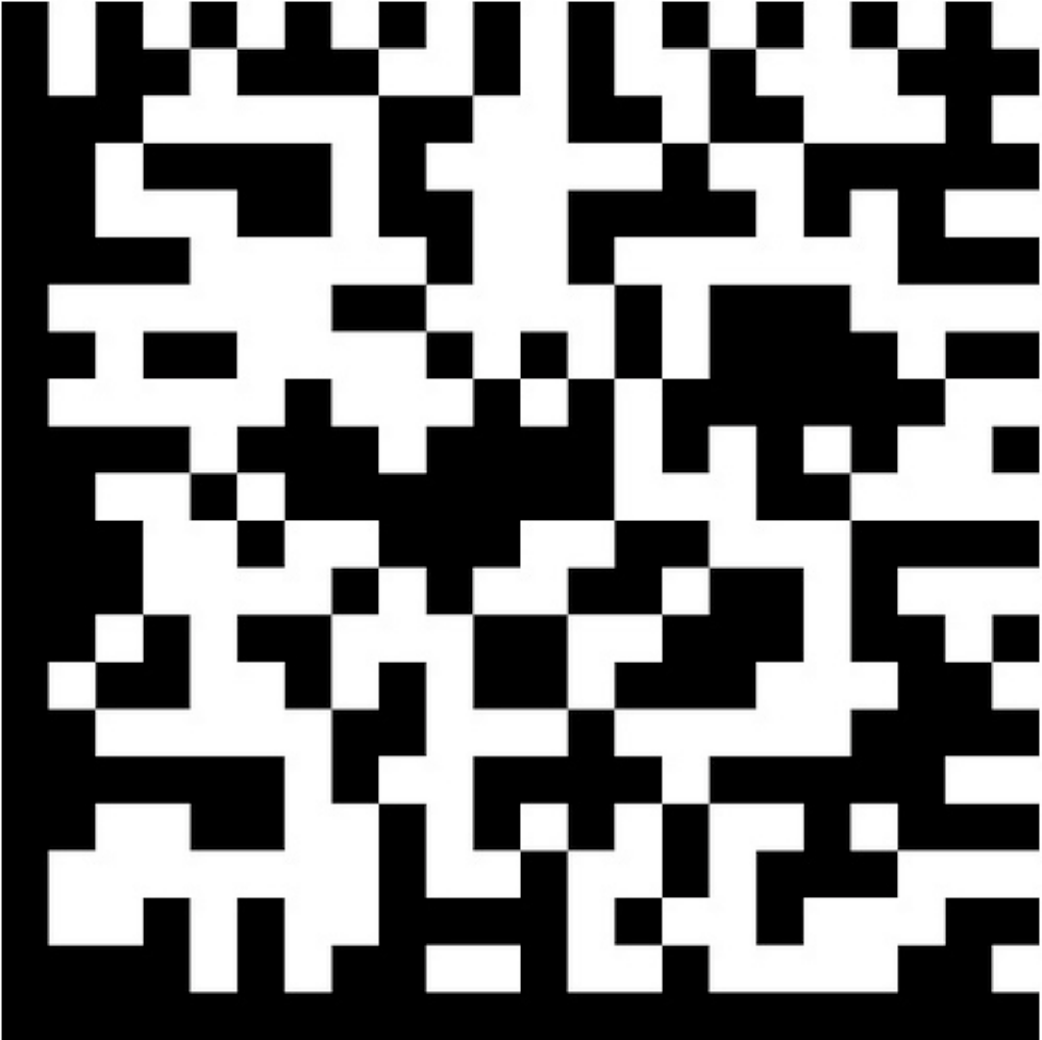}} \qquad \,\,\,\,{\includegraphics[viewport = 45 55 430 420,  width=1.47in]{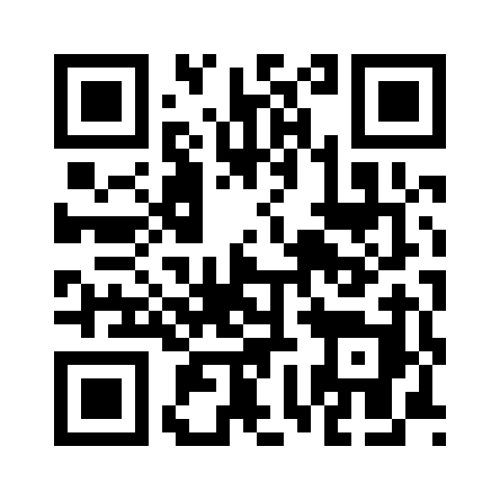}}} 
\caption{Examples of 2D bar codes: on the left is a {\it stacked} bar code while the other two are examples of {\it matrix} barcodes. The matrix code on the right is an example of a {\it QR (Quick Response)} barcode designed to be {\it readable}  with  camera-equipted smart phones.}
\label{fig:1}
\end{figure}

Here we consider three functionals based upon  a  particular anisotropic total variation well suited to 2D barcodes. It is defined for all $u \in BV(\R^2)$ (cf. \pref{TV}) as 
\begin{equation}\label{eq:anisotv}
\int_{\R^2} |u_x| + |u_y| := \sup \left\{ \int_{\R^2} u \,\text{div} v : v\in C_c^1(\R^2; \R^2), \forall x \, \, |v(x)|_\infty\leq 1\right\}, 
\end{equation}
where for a vector $v(x)= (v_1(x), v_2(x)) \in \R^2$, the norm $|\cdot|_\infty$ is defined by 
\[ |v(x)|_\infty \, : =\, \max \{ |v_1(x)|, |v_2(x)| \}. \]
This notation should not  be confused with the $L^\infty$ norm of a vector field $v\in L^\infty(\R^2; \R^2)$, denoted by 
$\|v\|_\infty$, which is the supremum of $|v(x)|$ over all $x\in \R^2$ where $|\cdot|$ denotes the standard Euclidean norm. 

Let $f \in L^1 (\R^2)$ denote an observed signal and $\lambda \ge 0$ be the so-called fidelity parameter. We consider the following functionals.
\begin{itemize}
\item \emph{Denoising}:
\[
F_1(u) := 
\int_{\R^2} |u_x| + |u_y| + \lambda \|u-f\|_{L^1(\R^2)}   \quad 
\hbox{\rm defined for }  u\in BV(\R^2). 
\]
Minimizing this functional includes no deblurring effects but the regularizing effect from the anisotropic total variation will lead to denoising.
\item \emph{Denoising} and slight \emph{deblurring}:
\[
\overline{F_1}(u) := 
\int_{\R^2} |u_x| + |u_y| + \lambda \|u-f\|_{L^1(\R^2)}  \quad  
\hbox{\rm defined for } u\in BV(\R^2; \{0, 1\}). 
\]
Note here the smaller domain of  binary functions which entails a  very crude attempt at deblurring. 
Whereas the  functional is no longer convex over its domain,  it has a convenient convex reformulation (cf. Lemma \ref{lem:convexify}) 
via the following functional: 
\[ 
{F_2} (v): =  \int_{\R^2} |v_x| + |v_y| \,\, + \,\, \lambda \int_{\R^2} 
\left ( | 1 - f| \, -\,  |f|\right) \, v 
 \quad 
\hbox{\rm defined for } v\in BV(\R^2; [0,1]). 
\]
\item \emph{Deblurring} and \emph{denoising}:
\[
F_3(u) := 
\int_{\R^2} |u_x| + |u_y| + \lambda \|Ku-f\|_{L^1(\R^2)} \quad 
\hbox{\rm defined for } u\in BV(\R^2). \]
Here
 \begin{equation}\label{eq:whatisK}
K: L^1(\R^2) \to L^1(\R^2) \text{ is a positive normalized}\footnote{i.e. $\int_{\R^2} Ku = \int_{\R^2} u$}\text{ bounded linear operator},
\end{equation}
e.g. convolution with a suitable blurring kernel which is commonly referred to ({\it cf.} \cite{ChanShen05}) as the  PSF ({\it point spread function}).
Our main interest here is  when  $f$ is of the form $f=Ku_0$, where $u_0 \in BV(\R^2)$, and in particular where $u_0$ is the characteristic function of a bar code (we could hence consider the larger class of operators $K$ defined on $BV(\R^2)$).  
The linear operator $K$ models the blurring of the bar code signal and  deblurring is introduced by the action of $K$ on $u$ prior to  comparison with  $f$. 


\end{itemize}
These functionals are variations of the original ROF functional with the following modifications:  
(i) The standard isotropic total variation of a characteristic function $u\in BV(\R^2, \{0,1\})$ gives the length of the perimeter of the set $\{u=1\}$. Using this as regularization term will lead to a rounding off of corners in the end result, because a rounded off corner has less interface than a sharp one. Our aim is to recover bar codes with sharp corners and hence we use this particular  anisotropic total variation in $F_i$ whose  corresponding Wulff shape (e.g. \cite{Taylor78}, \cite{EsedogluOsher04}) is a square. The anisotropic total variation\footnote{
Our choice of an anisotropic total variation does have a significant drawback: with the anisotropy along the coordinate axes, it assumes that the measured bar code is aligned with the coordinate axes (see also the definition of bar code in Section~\ref{sec:Q2}). In practice this assumption does not necessarily hold and the bar code might be rotated or even seen from a skewed perspective. Either a preprocessing step which aligns the bar code with the axes or the use of a rotated form of the anisotropic total variation might be in order (e.g. \cite{BerkelsBurgerDroskeNemitzRumpf06, ChuYangChen07b, XuMcCloskey11}).}
 we use gives, for a characteristic function $u$, the sum of the lengths of the projections of the perimeter of $\{u=1\}$ onto the coordinate axes. Moreover, it allows one to reformulate the discretized minimization problems in the form of linear programs which can be computationally solved  quickly and efficiently (see Section \ref{Numerical-Implementation}).  (ii) For the fidelity term we use the  $L^1$ distance.  As addressed in \cite{ChanEsedoglu05} the use of an $L^2$ fidelity term, as in the ROF functional, leads to a loss of contrast when minimizing over all of $BV(\R^2)$. 

From the point of view of image processing, the usefulness of this variational approach lies in the ability to denoise and deblur signals via minimization of the functionals. This has the potential to work well if the functionals are in fact faithful to the underlying images sought in the sense that if we input a perfect signal,  minimization  of the appropriate functional  will indeed yield back the perfect signal. 
We will say a functional $F_i$, $i \in \{1,2\}$ is \textit{faithful} to a 
signal $f$ {\it if} there exists some explicit regime for $\lambda$ such that the signal $f$ is the unique minimizer of $F_i$. We say  $F_3$ is {faithful} to a signal $u_0$ {\it if} there exists some explicit regime for $\lambda$ such that $u_0$ is the unique minimizer of $F_3$ with $f = Ku_0$. 
We thus focus on the following four  questions:
\begin{enumerate}
\item If the parameter $\lambda$ is chosen too small 
the lack of enforced fidelity to the measured signal will lead to the trivial minimizer $u=0$. For which values of $\lambda$ is this the case?

\item\label{item:whenbarcode} Are the functionals faithful to 
a clean 2D bar code and what are the associated values of 
 $\lambda$? 
How do these threshold values for $\lambda$ depend on the $X$-dimension of the bar code and the properties of the blurring operator $K$?

\item\label{item:onlybarcode} Are the $F_i$ faithful to other binary signals? 
 This is particularly relevant to judge the denoising properties of our functionals: if only clean 2D bar code signals are returned unchanged by minimization, this is an indication that noisy bar code signals will be denoised.
\begin{erratum}\label{err:onlyfaithful}
Earlier versions of this paper, including the version published in {\em Inverse Problems and Imaging} 5(3), 2010, pp. 591-617, erroneously argued that the only binary signals $F_1$ are $F_3$ are faithful to are clean 2D bar codes. A mistake in our argument, and in fact a counterexample for $F_1$, was pointed out to us by Matthias R\"oger and Nils Dabrock from the Technische Universit\"at Dortmund in April 2016. Starting from version arXiv:1007.1035v3 of this paper, Sections~\ref{sec:Q2} and~\ref{sec:minF3} have the incorrect results removed and instead contain errata explaining the mistake.
\end{erratum}

\item\label{item:numerics} What do numerical simulations for minimization of these functionals yield? Particularly,  how do these minimization algorithms perform in the presence of noise?

\end{enumerate} 

Question 1 is answered in Lemma \ref{trivial}. The basis for answering question
lies in the existence of a certain vector field  
(cf.  the definition of ${\mathcal V}$ in  \pref{eq:setofvfields}). Following an argument initially presented in \cite{Meyer01}, and elaborated on in  \cite{ChanEsedoglu05}, we find that a sufficient condition  for $u_0$ to be a minimizer is the existence of a vector field $v \in {\mathcal V}(u_0)$ (cf. Theorem~\ref{thm:minimizervectorfield}). Arguments from convex analysis (\cite{EkelandTemam76}, \cite{Ring00}) show that the existence of such a vector field is  also a  necessary condition (cf. Lemma \ref{lem:minimimpliesvectorf}). We then use the sufficiency,  and a particular  vector field construction, to show that  if $\lambda > \frac4{\omega}$,  a barcode $z$ is the unique minimizer of $F_1$ and $F_2$ with $f = z$ (cf. Corollaries~\ref{cor:leavebar codealone} and \ref{F2-barcode}). If $K$ represents convolution with any positive kernel (PSF) of unit mass, the same condition on $\lambda$ insures that $z$ is the unique minimizer of $F_3$ 
(cf. first part of Theorem~\ref{barcode-conv}). 
Question 4 is discussed in Sections~\ref{Numerical-Implementation} and~\ref{sec:numerical}.

Because of the importance these vector fields play in our analysis we introduce special notation for them (compare with the \emph{extremal pairs} in \cite[Section 1.14, Proposition 5]{Meyer01}). For a fixed $u\in BV(\R^2)$  we define 
\begin{equation}\label{eq:setofvfields}
\mathcal{V}(u) := \left\{v\in L^\infty(\R^2; \R^2): \text{ the three conditions below are satisfied}\right\},
\end{equation}
where the three conditions are
\begin{enumerate}
	\item \label{item:condition1} $|v(x)|_\infty \leq 1$ for almost all $x\in \R^2$,
	\item \label{item:condition2} $\dvg v \in L^\infty(\R^2)$,
	\item \label{item:condition3} $-\int_{\R^2} u\, \dvg v = \int_{\R^2} |u_x| + |u_y|$,
\end{enumerate}

\bigskip

There is an increasingly large literature on the analysis of these types of functionals ({\it cf.} \cite{ChanShen05}) and indeed, our work is highly guided by similar work of Chan,  Esedo\=glu,  Meyer, Osher,  Ring, and others \cite{RudinOsherFatemi92-2, Meyer01, EsedogluOsher04,ChanEsedoglu05, Ring00}. We are unaware of any analysis on this particular combination of $L^1$ fidelity and anisotropic total variation.

\section{Trivial minimizer}
We first state an elementary lemma involving the operator $K$. Its proof easily follows by decomposing $u$ into its positive and negative parts and using the triangle inequality. 
\begin{lemma}\label{lem:K}
Let $u\in L^1(\R^2)$, $K$ be as in (\ref{eq:whatisK}), 
and $\Omega\subset \R^2$ be open, 
then
\[
\left| \int_{\Omega} u \right| \leq \int_{\Omega} |Ku| \leq \int_{\Omega} |u|.
\]
\end{lemma}
\begin{lemma}\label{trivial}
Let there be an $R>0$ such that $\supp\, f \subset B(0;R)$ and define $\lambda_0 := \frac{C}{R\sqrt{\pi}}$ where $C$ is the isoperimetric constant from Lemma~\ref{lem:isoperimetric}. If $0\leq \lambda < \lambda_0$ then $u=0$ is the unique minimizer of $F_1$ and $\overline{F_1}$. If in addition $\ker K = \{0\}$ then $u=0$ is the unique minimizer of $F_3$ as well.
\end{lemma}
\begin{proof}
The proof is essentially the same as that of \cite[Proposition 5.7]{ChanEsedoglu05}. 
Since it is brief, we present it for $F_3$. Let $u$ be a minimizer of $F_3$ and hence $F_3(u) \leq F_3(0)$. 
By Lemma~\ref{lem:isoperimetric} we find
\begin{align*}
&\hspace{0.5cm}C \left(\int_{B(0;R)} u^2\right)^\frac12 + \lambda \int_{B(0;R)} |Ku-f| +  \lambda \int_{B(0;R)^c}|Ku|\\
&\leq C \left(\int_{\R^2} u^2\right)^\frac12 + \lambda \int_{\R^2} |Ku-f|
\leq \int_{\R^2} |u_x| + |u_y| + \lambda \int_{\R^2} |Ku-f|\\
&\leq F_3(0)
=  \lambda \int_{\R^2} |f|
= \lambda \int_{B(0;R)} |f|
\leq \lambda \int_{B(0;R)} |Ku-f| + \lambda \int_{B(0;R)} |Ku|.
\end{align*}
Lemma~\ref{lem:K} and H\"older's inequality tell us that
\[
\int_{B(0;R)} |Ku| \leq \int_{B(0;R)} |u| \leq R \sqrt{\pi} \left(\int_{B(0;R)} u^2\right)^\frac12.
\]
Applying this to the left hand side of the calculation above we get
\[
(\lambda_0 - \lambda) \int_{B(0;R)} |Ku| +  \lambda \int_{B(0;R)^c}|Ku| \leq 0.
\]
Since $\lambda < \lambda_0$,  $\int_{\R^2} |Ku| = 0$.
\end{proof}

\section{Minimizers of $F_1$}
We first explore the consequence of the simple property in  convex analysis that  (cf. \cite{EkelandTemam76})  for a convex functional $F$ defined over a topological vector space $V$,  
\[ u_0  \,\, \hbox{\rm is a minimizer of   } F  \hbox{ \rm over }  V \quad    \iff\quad 0 \in \partial F(u_0), \]
where the subdifferential of $F$ at $u_0$ ($F(u_0) < \infty$)  is defined by  
\[ \partial F(u_0) := \left\{ u^\ast \in V^\ast : \,\, 
\forall \, v \in  V, \,\, {_{V^*}}\langle u^*, v - u_0 \rangle_V \,\, + \,\, F(u_0) \, \le \, F(v) \right\}. \]
Here $V^\ast$ denotes the topological dual space with pairing ${_{V^*}}\langle \cdot, \cdot \rangle_V$.

There are some subtleties involved in choosing the right space $V$ to define our functional(s) on. Our choice here is to  take $V = BV(\R^2)$. This has the advantage that it allows for some basic subdifferential calculus but  forces us to work with the dual space $BV(\R^2)^\ast$ and its associated pairing with $BV(\R^2)$ which both lack a simple general explicit description. However, we only need this dual space structure in a specific case in which we are able to give an explicit description of the dual element and its action on smooth functions in $BV(\R^2)$, see (\ref{eq:dualpairingpsi}).  An alternative approach would be to take $V=L^2(\R^2)$ after extending all our functionals to $L^2(\R^2)$ by setting their value to be $+\infty$ on $L^2(\R^2) \backslash BV(\R^2)$. For the anisotropic total variation by itself, this would work well and for example, its subgradient  over $V=L^2(\R^2)$ was calculated in \cite[Theorem 12]{Moll05}. 
However, the $L^1$ fidelity term then lacks the necessary continuity properties to be treated separately and we would need to compute the subgradient of the functional as a whole. Defining the functionals over $L^1(\R^2)$ would solve this issue, but would still leave us with a complication in the computation of the subgradient of the anisotropic total variation because the gradient operator is not continuous on $L^1(\R^2)$.


We introduce some notation.  Let $\mathcal{M}_2(\R^2)$ denote the set of all vector valued Radon measures on $\R^2$ with two components. Denote the transpose of the gradient operator
\[
\nabla: BV(\R^2) \to \mathcal{M}_2(\R^2), u \mapsto (\partial_1 u, \partial_2 u)
\]
by
\[
\nabla^*: \mathcal{M}_2(\R^2)^* \to BV(\R^2)^*.
\]
Note that since $L^1(\R^2; \R^2)\subset \mathcal{M}_2(\R^2)$ we have $\mathcal{M}_2(\R^2)^* \subset L^\infty(\R^2; \R^2)$. For a $v\in \mathcal{M}_2(\R^2)^*$ and $\mu \in \mathcal{M}_2(\R^2)$ we can write the coupling as $\int_{\R^2} v\,d\mu$. In particular if $\mu$ is absolutely continuous with respect to the Lebesgue measure we can identify $\mu$ via the Radon-Nikodym derivative with a function in $L^1(\R^2; \R^2)$ which we again denote by $\mu$. In that case the coupling with $v$ can be written as $\int_{\R^2} v\cdot \mu \,d\mathcal{L}^2$ (we will leave $d\mathcal{L}^2$ out where there is no confusion).


Furthermore we introduce
\[
\|\cdot\|_a: \mathcal{M}_2(\R^2) \to \R, \mu \mapsto \sup\left\{ \int_{\R^2} v_1 d\mu_1 + \int_{\R^2} v_2 d\mu_2: v\in C_c^\infty(\R^2),  \forall x, \, |v(x)|_\infty \leq 1\right\}
\]
(where we interpret $\int_{\R^2} v_i \, d\partial_i u = - \int_{\R^2} u \partial_i v_i \,dx$ for $u\in BV(\R^2)$). Note that
\[
\|\nabla u\|_a = \int_{\R^2} |u_x| + |u_y|.
\]

\begin{lemma}\label{lem:minimimpliesvectorf}
If $u_0\in BV(\R^2)$ is a minimizer of $F_1$ over $BV(\R^2)$, then there exists a vector field $v\in \mathcal{V}(u_0)$ and $\lambda \geq \|\dvg v\|_{L^\infty(\R^2)}$.
\end{lemma}

\begin{proof}
This proof follows similar lines as the proofs of \cite[Theorem 2.3]{CasasKunischPola99} and \cite[Proposition 3]{Ring00}. Note that the result is trivially true if $u_0=0$.
We explore the consequences of    $0 \in \partial F_1(u_0)$.  Because both terms in $F_1$ are continuous in $u$, the subdifferential $\partial F(u_0)$ is given by the sum of the subdifferentials of the separate terms \cite[Chapter I, Proposition 5.6]{EkelandTemam76}.

The functional $u \mapsto \int_{\R^2} |u_x|+|u_y|$ can be written as the composition: $\|\cdot\|_a \circ \nabla$. Because $\|\cdot\|_a$ is a continuous functional on $\mathcal{M}_2(\R^2)$ 
and $\nabla$ is a continuous linear mapping from $BV$ to $\mathcal{M}_2(\R^2)$,  we can apply \cite[Chapter I, Proposition 5.7]{EkelandTemam76}:
\[
\partial(\|\cdot\|_a \circ \nabla)(u_0) = \nabla^* \partial \|\cdot\|_a(\nabla u_0).
\]
This means that
\[
\chi \in \partial(\|\cdot\|_a \circ \nabla)(u_0)
\Leftrightarrow \exists v \in \partial \|\cdot\|_a(\nabla u_0) \text{ such that } \chi = \nabla^* v.
\]
Note that since $L^1(\R^2; \R^2)\subset \mathcal{M}_2(\R^2)$ we have $\partial \|\cdot\|_a(\nabla u_0) \subset \mathcal{M}_2(\R^2)^* \subset L^\infty(\R^2; \R^2)$ and hence $v\in L^\infty(\R^2; \R^2)$.
Because $u_0\in BV(\R^2)$ and hence $\|\nabla u_0\|_a$ is finite, the subdifferential $\partial \|\cdot\|_a(\nabla u_0)$ is characterized by
\begin{equation}\label{eq:conditiononv}
v \in \partial \|\cdot\|_a(\nabla u_0) \Leftrightarrow \forall \mu \in \mathcal{M}_2(\R^2)\quad  {_{\mathcal{M}_2(\R^2)^*}}\langle v, \mu - \nabla u_0 \rangle_{\mathcal{M}_2(\R^2)} + \|\nabla u_0\|_a \leq \|\mu\|_a.
\end{equation}

Next we turn to the subdifferential for the fidelity term in $F_1$. First note that since $W^{1, 1}(\R^2) \subset BV(\R^2)$ we have $\partial \left(\int_{\R^2} |\cdot-f|\right)(u_0) \subset BV(\R^2)^* \subset W^{1, 1}(\R^2)^*$. Since the fidelity term in $F_1$ is finite for $u=u_0\in BV(\R^2)$, the subdifferential is determined by 
\begin{equation}\label{eq:conditiononpsi}
\psi \in \partial \left(\int_{\R^2} |\cdot-f|\right)(u_0) \Leftrightarrow \forall u\in BV(\R^2)\quad {_{BV(\R^2)^*}}\langle \psi, u-u_0 \rangle_{BV(\R^2)} + \int_{\R^2} |u_0-f| \leq \int_{\R^2} |u-f|.
\end{equation}

We deduce that there exist $v$ and $\psi$ as above such that
\begin{equation}\label{eq:nablastarvlambdapsi}
\nabla^* v + \lambda \psi = 0.
\end{equation}
Choosing $\mu=0$ and $\mu = 2 \nabla u_0$ in the right hand side statement of (\ref{eq:conditiononv}) leads to
\[
-{_{\mathcal{M}_2(\R^2)^*}}\langle v,\nabla u_0 \rangle_{\mathcal{M}_2(\R^2)} \leq -\|\nabla u_0\|_a \quad \text{and} \quad {_{\mathcal{M}_2(\R^2)^*}}\langle v,\nabla u_0 \rangle_{\mathcal{M}_2(\R^2)} \leq \|\nabla u_0\|_a,
\]
hence
\begin{equation}\label{eq:measurenablaf}
{_{\mathcal{M}_2(\R^2)^*}}\langle v,\nabla u_0 \rangle_{\mathcal{M}_2(\R^2)} = \|\nabla u_0\|_a.
\end{equation}
Substituting this back into (\ref{eq:conditiononv}) gives, for all $\mu\in \mathcal{M}_2(\R^2)$,
\[
{_{\mathcal{M}_2(\R^2)^*}}\langle v, \mu \rangle_{\mathcal{M}_2(\R^2)} \leq \|\mu\|_a.
\]
If we restrict this to $\mu=(\mu_1, \mu_2) \in L^1(\R^2; \R^2)$ this reads
\[
\int_{\R^2} v\cdot \mu \leq \int_{\R^2} |\mu_1| + |\mu_2|
\]
and hence $|v(x)|_\infty \leq 1$ almost everywhere, i.e. condition~\ref{item:condition1} in (\ref{eq:setofvfields}) holds.

Let $\tilde u\in BV(\R^2)$ and choose $u=\pm \tilde u + u_0$ in the right hand side statement of (\ref{eq:conditiononpsi}) (and then drop the tilde), then we get that for all $u \in BV(\R^2)$
\begin{equation}\label{eq:psiulequ}
\left| {_{BV(\R^2)^*}}\langle \psi, u \rangle_{BV(\R^2)} \right| \leq \int_{\R^2} |u|.
\end{equation}
We  use~\pref{eq:psiulequ}  to show that there exists an $L^\infty(\R^2)$ function such that the action of $\psi$ on test functions is given by integration against this function.
To this end, note that  $\psi \in W^{1,1}(\R^2)^*$ and hence  by 
\cite[3.4 and Theorem 3.8]{Adams75},  there exist $(p, q) \in L^\infty(\R^2) \times L^\infty(\R^2; \R^2)$ such that for all $u\in W^{1,1}(\R^2)$
\[
{_{BV(\R^2)^*}}\langle \psi, u \rangle_{BV(\R^2)} = \int_{\R^2} p\, u + \int_{\R^2} q\cdot \nabla u.
\]
Fix $u\in C_c^\infty(\R^2)$ and $y\in \R^2$. For $\e>0$ define $u_\e(x) := \e^{-2} u(\e^{-1} (x-y))$ and use this as $u$ in (\ref{eq:psiulequ}) to find via a substitution of variables
\[
\left| \int_{\R^2} p(\e x+y)\, u(x) \,dx + \e^{-1} \int_{\R^2} q(\e x+y)\cdot \nabla u(x) \,dx\right| \leq \int_{\R^2} |u(x)| \,dx.
\]
Because
\begin{align*}
\left| \int_{\R^2} p(\e x+y)\, u(x) \,dx \right| &\leq \|p\|_{L^{\infty}(\R^2)} \|u\|_{L^1(\R^2)},\\
\left|\int_{\R^2} q(\e x+y)\cdot \nabla u(x) \,dx\right| &\leq \|q\|_{L^\infty(\R^2)} \|\nabla u\|_{L^1(\R^2)},
\end{align*}
taking the limit $\e\to 0$ we deduce that
\[
\underset{\e\to 0}{\lim}\, \int_{\R^2} q(\e x+y)\cdot \nabla u(x) \,dx = 0.
\]
We claim that the above implies $q = 0$ a.e. To see this note that 
since $u$, and hence $\nabla u$, have compact support, we can replace, for $\e$ small,  $q$ by $\hat q := q|_{y+\supp u}$ in the preceding integral. We have, for any $p\geq1$, $\hat q \in L^p(\R^2; \R^2)$ and hence there exists a sequence $\{\hat q_n\}_{n=1}^\infty \subset C_c^\infty(\R^2; \R^2)$ such that $\hat q_n \to \hat q$ in $L^p(\R^2; \R^2)$. Via H\"older's inequality we find that
\begin{equation}\label{eq:bunchofintegrals}
0 = \underset{\e\to 0}{\lim}\, \int_{\R^2} q(\e x+y)\cdot \nabla u(x) \,dx = \underset{\e\to 0}{\lim}\, \underset{n\to \infty}{\lim}\, \int_{\R^2} \hat q_n(\e x+y)\cdot \nabla u(x) \,dx.
\end{equation}
Since $\{\hat q_n\}_{n=1}^\infty$ is a compact set of continuous functions defined on a compact set by the Arzel\`a-Ascoli theorem it is equicontinuous and hence
\[
\int_{\R^2} \hat q_n(\e x+y)\cdot \nabla u(x) \,dx \to \int_{\R^2} \hat q_n(y)\cdot \nabla u(x) \,dx, \quad \text{uniformly in } n \text{ as } \e \to 0.
\]
This allows us to interchange the limits in (\ref{eq:bunchofintegrals}) and find
\[
\underset{\e\to 0}{\lim}\, \underset{n\to \infty}{\lim}\, \int_{\R^2} \hat q_n(\e x+y)\cdot \nabla u(x) \,dx =  \underset{n\to \infty}{\lim}\, \underset{\e\to 0}{\lim}\, \int_{\R^2} \hat q_n(\e x+y)\cdot \nabla u(x) \,dx.
\]
We  now apply the dominated convergence theorem to find that 
\[
\underset{n\to \infty}{\lim}\, \underset{\e\to 0}{\lim}\, \int_{\R^2} \hat q_n(\e x+y)\cdot \nabla u(x) \,dx = \underset{n\to \infty}{\lim}\,  \int_{\R^2} \hat q_n(y)\cdot \nabla u(x) \,dx.
\]
Because $L^p$ convergence implies pointwise convergence almost everywhere and $\hat q_n(y)$ does not depend on the integration variable $x$, we conclude that for almost every $y\in \R^2$
\[
\underset{n\to \infty}{\lim}\,  \int_{\R^2} \hat q_n(y)\cdot \nabla u(x) \,dx = \int_{\R^2} \hat q(y)\cdot \nabla u(x) \,dx = \int_{\R^2} q(y)\cdot \nabla u(x) \,dx=0.
\] 
The function $u$ was chosen arbitrarily, hence $q=0$ a.e.
and we conclude for all $u\in C^\infty_c(\R^2)$, 
\begin{equation}\label{eq:dualpairingpsi}
{_{BV(\R^2)^*}}\langle \psi, u \rangle_{BV(\R^2)} = \int_{\R^2} p\, u.
\end{equation}
We thus have  $p=\psi$ in the sense of distributions, and hence 
 by equation (\ref{eq:nablastarvlambdapsi}),  there exists  $r \in L^\infty(\R^2)$ which as a distribution may be identified with 
$\nabla^* v$ such that  
\[
{_{BV(\R^2)^*}}\langle \nabla^* v, u \rangle_{BV(\R^2)} = \int_{\R^2} r \, u.
\]
Since for every $\varphi\in C_c^\infty(\R^2)$, 
\begin{align}
\int_{\R^2} \dvg v\, \varphi &= -\int_{\R^2} v \cdot \nabla \varphi = - {_{\mathcal{M}_2(\R^2)^*}}\langle v, \nabla \varphi\rangle_{\mathcal{M}_2(\R^2)}\notag\\
&=  -{_{BV(\R^2)^*}}\langle \nabla^* v, \varphi\rangle_{BV(\R^2)} = -\int_{\R^2} r  \,\varphi,\label{eq:nablastarisminusdiv}
\end{align}
we have  $-\dvg v = r \in L^\infty (\R^2)$ in the sense of distributions.
This gives  condition~\ref{item:condition2} in (\ref{eq:setofvfields}). Also by equation (\ref{eq:nablastarvlambdapsi}):
$\displaystyle
-\dvg v + \lambda \psi = 0, 
$
in the sense of distributions. 

We remember that we can write (\ref{eq:measurenablaf}) as
\[
\int_{\R^2} v\cdot d\nabla u_0= \int_{\R^2} |u_{0_x}| + |u_{0_y}|.
\]
For $\phi \in C_c^1(\R^2; \R^2)$ we have, \cite[equation (3.2)]{AmbrosioFuscoPallara00},
\[
\int_{\R^2} \phi\cdot d\nabla u_0 = -\int_{\R^2} u_0\, \dvg \phi.
\]
The vector field $v$ satisfies all the conditions of Lemma~\ref{lem:approxv} and hence there exists a sequence $\{v_j\}_{j=1}^{\infty} \subset C_c^\infty(\R^2; \R^2)$ such that as $j\to\infty$, $v_j \overset{*}\rightharpoonup v$ in $L^\infty(\R^2; \R^2)$, and $\dvg v_j  \overset{*}\rightharpoonup \dvg v$ in $L^\infty(\R^2)$. We compute
\[
-\int_{\R^2} u_0\, \dvg v = -\underset{j\to\infty}{\lim}\,\int_{\R^2} u_0\, \dvg v_j = \underset{j\to\infty}{\lim}\,\int_{\R^2} v_j \cdot d\nabla u_0.
\]
From $\|v_j\|_{L^\infty(\R^2)} \leq C$ and $v_j \in C_c^\infty(\R^2)$ we deduce that $\underset{x\in\R^2}{\max}\, v_j(x) \leq C$. Since $u_0\in BV(\R^2)$ we have that $-\infty < \nabla u_0(\R^2) < \infty$ and hence the constant function $C$ is integrable against the measure $\nabla u_0$. Because $v_j(x) \to v(x)$ almost everywhere the dominated convergence theorem tells us
\[
\underset{j\to\infty}{\lim}\,\int_{\R^2} v_j \cdot d\nabla u_0 = \int_{\R^2} v \cdot d\nabla u_0,
\]
from which we conclude that condition~\ref{item:condition3} in (\ref{eq:setofvfields}) holds.

Finally, combining (\ref{eq:psiulequ}), (\ref{eq:nablastarvlambdapsi}), and the fact that $\psi$ as distribution is represented by an $L^\infty(\R^2)$ function, we have by density of $C_c^\infty(\R^2)$ in $L^1(\R^2)$ that for every $u\in L^1(\R^2)$
\begin{equation}\label{eq:getlambdabound}
\int_{\R^2} u\, \dvg v = \lambda \int_{\R^2} u\, \psi \leq \lambda \int_{\R^2} |u|, 
\end{equation}
and therefore $\lambda \geq \|\dvg v\|_{L^\infty(\R^2)}$.

\end{proof}

\begin{theorem}\label{thm:minimizervectorfield} 
Let $f\in BV(\R^2)$ be the measured signal in $F_1$, 
then the following two statements are equivalent:
\begin{enumerate}
\item\label{item:fisminimizerF1} $f$ is a minimizer of $F_1$ over $BV(\R^2)$.
\item\label{item:vectorfield} There exists a vector field $v\in \mathcal{V}(f)$ and $\lambda \geq \|\dvg v\|_{L^\infty(\R^2)}$.
\end{enumerate}
Moreover, if there exists a vector field $v\in \mathcal{V}(f)$ and $\lambda > \|\dvg v\|_{L^\infty(\R^2)}$, then $f$ is the unique minimizer of $F_1$ over $BV(\R^2)$.
\end{theorem}
\begin{proof}

The implication \ref{item:fisminimizerF1} $\Rightarrow$ \ref{item:vectorfield} follows directly from Lemma~\ref{lem:minimimpliesvectorf}. For the reverse direction, we follow \cite[Lemma 5.5]{ChanEsedoglu05}. 
To this end, for any $u\in BV(\R^2)$  Corollary~\ref{cor:anisotvLinfty} implies 
\[
\int_{\R^2} |u_x| + |u_y| \geq -\int_{\R^2} u\, \dvg v
\]
and hence
\begin{align*}
F_1(u) &\geq 
-\int_{\R^2} u\, \dvg v + \lambda \int_{\R^2} |u-f|\\
&= -\int_{\R^2} f\, \dvg v + \lambda \int_{\R^2} |u-f| - \int_{\R^2} (u-f)\, \dvg v\\
&\geq F_1(f) + \big( \lambda - \|\dvg v\|_{L^\infty(\R^2)}\big) \int_{\R^2} |u-f|\\
&\geq F_1(f).
\end{align*}
If $\lambda > \|\dvg v\|_{L^\infty(\R^2)}$, the last inequality  is strict.
\end{proof}

\section{$F_1$ and 2D barcodes}\label{sec:Q2}

We define some further notation to make the idea of a 2D bar code precise. By a 2D bar code we mean a bounded set  $S\subset \R^2$  
such that $\partial S$ is the union of a finite number of non-self-intersecting polygonal boundaries, each a finite union of horizontal and vertical line segments.
For $s\in \R$, define the horizontal and vertical lines
\[
l_h(s) := \left\{x\in \R^2: x_2=s\right\}, \qquad l_v(s) := \left\{ x\in \R^2: x_1 = s\right\}.
\]
For given $s\in \R$ let $\omega_h^1(s)$ be the width of the shortest connected component of $l_h(s) \cap S$ and $\omega_h^2(s)$ the width of the shortest connected component of $l_h(s)\cap S^c$. Analogously define $\omega_v^{1,2}(s)$ for the connected components of $l_v(s)\cap S$ and $l_v(s)\cap S^c$. We define the $X$-dimension of $S$ to be 
\[
\omega := \underset{s\in\R}{\min}\, \{\omega_h^1(s), \omega_h^2(s), \omega_v^1(s), \omega_v^2(s)\}.
\] 
In words: The $X$-dimension is the shortest horizontal and vertical length scale of both the black squares and the white background.
We denote the set of 2D bar codes by $\mathcal{S}$ and the set of 2D bar codes with prescribed $X$-dimension larger than or equal to $\omega$ by $\mathcal{S}_\omega$. 
Note that we do not require the horizontal and vertical length scales of the black squares or the white background to be integer multiples of the $X$-dimension.
We identify the 2D bar codes with their characteristic functions ($\chi_S$ is the characteristic function of the set $S$):
\begin{align}
\mathcal{B} &:= \left\{ u \in BV(\R^2; \{0, 1\}) : \text{ there is an } S \in \mathcal{S} \text{ such that } u = \chi_S\right\},\label{eq:barcodes}\\
\mathcal{B}_\omega &:= \left\{ u \in BV(\R^2; \{0, 1\}) : \text{ there is an } S \in \mathcal{S}_\omega \text{ such that } u = \chi_S\right\}.\notag
\end{align}

%

By $\partial_*S$ we denote the {\em reduced boundary} of a set $S$, i.e. all points in $\partial S$ for which there is a well defined normal vector (see \cite[Definition 3.54]{AmbrosioFuscoPallara00}). For example, if $S$ is a square its reduced boundary consists of all boundary points except the corners.

\begin{theorem}\label{thm:vectorfieldforbar code}
Let $S \in \mathcal{S}_\omega$. Then there exists a vector field $v\in \mathcal{V}(\chi_S)$ with  $\|\dvg v\|_{L^\infty(\R^2)} = \frac4\omega$.

\end{theorem}

\begin{figure} 
\centerline{{\includegraphics[width=1.8in]{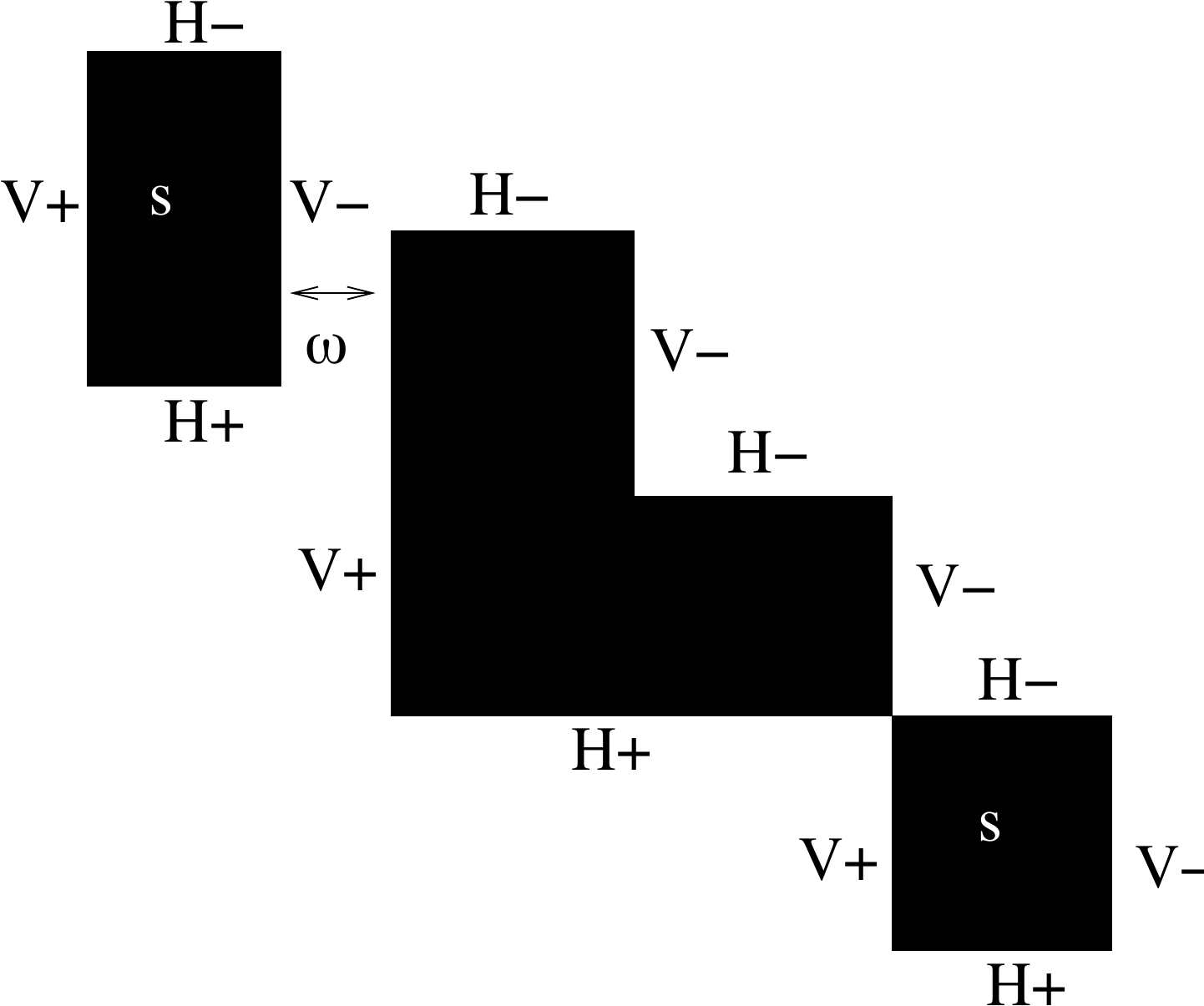}}}
\caption{A possible choice for S in Theorem~\ref{thm:vectorfieldforbar code} with the boundary parts $V_\pm$ and $H_\pm$ indicated as well as the length $\omega$.}
\label{fig:Sregion}
\end{figure}

\begin{proof}
Let $n(x)$ be the outward normal to $S$ at $x\in \partial_* S$ and define the following subsets of $\partial_* S$:
\[
V_\pm := \left\{ x\in \partial_* S: -n(x) = (\pm 1, 0)\right\}, \qquad H_\pm := \left\{ x\in \partial_* S: -n(x) = (0, \pm 1)\right\}.
\]
Let $V := V_- \cup V_+$ and $H:= H_- \cup H_+$  ($V\cup H = \partial_* S$). See Figure~\ref{fig:Sregion} for an illustration.
For each $s\in\R$, if $l_h(s) \cap V$ and $l_v(s) \cap H$ are nonempty, they are finite sets of isolated points. Furthermore, if we write $l_h(s) \cap V = \{x^{(1)}, \ldots x^{(2n)}\}$ where  $x_1^{(i)} < x_1^{(i+1)}$ for all $i$, then $x^{(i)} \in V_+$ for odd $i$ and $x^{(i)} \in V_-$ for even $i$. The analogous statement holds for $l_v(s) \cap H$. By definition of $\omega$ we have 
\begin{equation}
\forall x, y\in l_h(s) \cap V,\, |x_1-y_1| \geq \omega \quad {\rm and} \quad \forall x, y\in l_v(s) \cap H,\, |x_2-y_2| \geq \omega.\label{eq:distance}
\end{equation}
For each such $s\in \R$, fix $x^{(0)}, x^{(2n+1)} \in l_h(s)$ such that $|x_1^{(0)} - x_1^{(1)}| \geq \omega$ and $|x_1^{(2n)} - x_1^{(2n+1)}| \geq \omega$. 

We define $v_1 \in L^\infty(\R^2)$ with compact support as follows. For each $s\in\R$ for which $l_h(s) \cap V = \emptyset$, let $v_1(\cdot, s)=0$. For each $s\in \R$ for which $l_h(s) \cap V$ is not empty,  $v_1(\cdot,  s)$ is defined by: 
\begin{itemize}
\item linear on each interval $\left[x_1^{(i)}, x_1^{(i+1)}\right]$ for $i=0, \ldots, 2n$,
\item 
$ v_1(x) = \left\{ \begin{array}{ll} -1 & \text{if } x\in V_-,\\ 1 & \text{if } x\in V_+,\end{array}\right.$
\item  zero on $\left(-\infty, x_1^{(0)}\right] \cap \left[x_1^{(2n+1)}, \infty\right)$.
\end{itemize}
The function $v_1(\cdot, s)$ is Lipschitz continuous for each $s\in \R$ and $\left\|\frac{\partial v_1}{\partial x_1}\right\|_{L^\infty(\R^2)} = \frac2\omega$.\\
We define $v_2 \in L^\infty(\R^2)$ in a similar way on each vertical line $l_v(s)$. In particular $v_2(s, \cdot)$ is Lipschitz continuous for each $s\in R$ and satisfies
\[
v_2(x) = \left\{ \begin{array}{ll} -1 & \text{if } x\in H_-,\\ 1 & \text{if } x\in H_+,\end{array}\right. \qquad{\rm with} \quad \left\|\frac{\partial v_2}{\partial x_2}\right\|_{L^\infty(\R^2)} = \frac2\omega.
\]
Now for $v:=(v_1, v_2)^T \in L^\infty(\R^2; \R^2)$ we have $|v(x)|_\infty \leq 1$ for $x\in \R^2$. We compute
\begin{align*}
-\int_S \dvg v &= -\int_{\partial_* S} v\cdot n
= -\int_{V_-} v_1 - \int_{H_-} v_2 + \int_{V_+} v_1 + \int_{H_+} v_2\\
&= \int_{V_-} 1 + \int_{H_-} 1 + \int_{V_+} 1 + \int_{H_+} 1
= \mathcal{H}^1(\partial_*S)
= \mathcal{H}^1(\partial S)\\
&= \int_{\R^2} |\chi_{S_x}| + |\chi_{S_y}|.
\end{align*}
We have $\dvg v \in L^\infty(\R^2)$ with 
\[
\|\dvg v\|_{L^\infty(\R^2)} \leq  \left\|\frac{\partial v_1}{\partial x_1}\right\|_{L^\infty(\R^2)} + \left\|\frac{\partial v_2}{\partial x_2}\right\|_{L^\infty(\R^2)} \,= \, \frac4\omega.
\]

\end{proof}

\begin{corol}\label{cor:leavebar codealone}
Let $f\in \mathcal{B}_\omega$ be the measured signal in $F_1$. If $\lambda \geq \frac4\omega$ ($\lambda > \frac4\omega$), then $u=f$ is a minimizer (the unique minimizer) of $F_1$ over $BV(\R^2)$.
\end{corol}
\begin{proof}
This follows directly from Theorems~\ref{thm:minimizervectorfield} and~\ref{thm:vectorfieldforbar code}.
\end{proof}


\begin{remark}\label{rem:differentvectorfield}
In the case where $S$ is a single square one can adapt the vector field construction in \cite[Theorem 4.1]{EsedogluOsher04} to our situation to get a different vector field $\tilde v\in \mathcal{V}(\chi_S)$. In this case one finds the same bound on $\lambda$ since $\|\dvg \tilde v\|_{L^\infty(\R^2)} = \frac4\omega$.
\end{remark}

\begin{erratum}\label{err:counterexample}
Versions of this paper prior to arXiv:1007.1035v3 ended the current section with a lemma that stated that: if $f\in BV(\R^2; \{0,1\})$ is both the measured signal in $F_1$ and a minimizer of $F_1$ over $BV(\R^2)$, then $f\in \mathcal{B}$. This statement is false, as can be seen from the following counterexample, which was provided in a private communication by Nils Dabrock.

Let $h\in (1/\sqrt{2},1)$ and let $\Omega := B(0,1) \cap [-h,h]^2$ be a truncated circle. Let $f = \chi_\Omega \in L^1(\R)$ be the characteristic function of $\Omega$. Define $s:= \sqrt{1-h^2}$ and let\footnote{Numerical simulations by Nils Dabrock suggest this condition can be weakened to $\lambda>\frac1s$.} $\lambda > \frac2s$. Define, for $r\in \R$, $w(r):= \min\{1, \max\{-1,r/s\}\}$ and let, for $(x,y)\in \R^2$,
\[
v(x,y) := \left(\begin{array}{c} w(x)\\ w(y)\end{array}\right).
\]
We will now show that $v\in \mathcal{V}(f)$ and hence, by Theorem~\ref{thm:minimizervectorfield}, $F_1$ is faithful to $f$.

It can be verified by direct computation that, for $x\in \R^2$, $|v(x)|_\infty \leq 1$ and $\|\dvg v\|_{L^\infty(\R^2)} \leq \frac2s$. Furthermore, we have for all $z\in \partial\Omega$ that $v(z)\cdot n_{\partial \Omega}(z) = |n_{\partial \Omega}(z)|_1$, where $n_{\partial \Omega}$ is the outward normal vector to the boundary $\partial \Omega$. By the definition in (\ref{eq:anisotv}) and Corollary~\ref{cor:anisotvLinfty} we then find
\begin{align*}
\int_{\R^2} |f_x| + |f_y| &= \sup\left\{\int_{\R^2} f\, \dvg \varphi: \varphi \in C_c^1(\R^2; \R^2), \forall z \, \, |\varphi(z)|_\infty \leq 1\right\}\\
&= \sup\left\{\int_\Omega \dvg \varphi: \varphi \in C_c^1(\R^2; \R^2), \forall z \, \, |\varphi(z)|_\infty \leq 1\right\}\\
&= \sup\left\{\int_{\partial\Omega} \varphi \cdot n_{\partial \Omega}: \varphi \in C_c^1(\R^2; \R^2), \forall z \, \, |\varphi(z)|_\infty \leq 1\right\}\\
&\leq \int_{\partial\Omega} |n_{\partial\Omega}(z)|_1 = \int_{\partial\Omega} v\cdot n_{\partial\Omega} = \int_{\R^2} f\, \dvg v\\
&\leq \sup \left\{ \int_{\R^2} f \,\text{div} \,  \varphi : \varphi\in L^\infty(\R^2; \R^2), \dvg \varphi \in L^\infty(\R^2),  \, |\varphi(z)|_\infty\leq 1 \, {\rm a.e.} \right\}\\
&= \int_{\R^2} |f_x|+|f_y|.
\end{align*}

\end{erratum}

\section{Minimizers of $\overline{F_1}$ and $F_2$} 

Next we investigate minimizers of $\overline{F_1}$. Due to the binary constraint on  admissible functions, the minimization problem is no longer convex. However, following Chan and Esedo\=glu \cite{ChanEsedoglu05} (see also \cite{ChanEsedogluNikolova06}) there is a simple and elegant convex reformulation of  the problem. 
\begin{lemma}\label{lem:convexify}
If $v\in BV(\R^2; [0,1])$ is a minimizer of the convex functional 
\[{F_2} (v): = \int_{\R^2} |v_x| + |v_y| \,\, + \,\, \lambda \int_{\R^2} 
\left ( | 1 - f| \, -\,  |f|\right) \, v,
\]
then for a.e. $t\in [0,1]$, $u = \chi_{E(t)}$ 
is a minimizer of $F_1$ over $BV(\R^2; \{0, 1\})$ 
with the same $\lambda$ where 
\[ E(t) := \{x\in\R^2 :  v(x) \ge t\}.
\]
\end{lemma}
\begin{proof}
The proof is  essentially a repetition of \cite[Theorem 2]{ChanEsedogluNikolova06}.
Let $v\in BV(\R^2; [0,1])$ be a minimizer of ${F_2}$. Because
\[
v(x) = \int_0^1 \chi_{[0, v(x)]}(t)\,dt = \int_0^1 \chi_{E(t)}(x)\,dt
\]
we compute
\begin{align*}
\int_{\R^2} \big( |1-f| -|f| \big) v &= \int_{\R^2} \int_0^1 \big( |1-f| -|f| \big) \chi_{E(t)} \,dt \,dx
= \int_0^1 \int_{E(t)} \big( |1-f| - |f| \big) \,dx \,dt\\
&= \int_0^1 \left( \int_{E(t)} |1-f| - \int_{\R^2} |f| + \int_{E(t)^c} |f| \right)
= \int_0^1 \left( \int_{\R^2} |\chi_{E(t)}-f| \right) - C,
\end{align*}
where $C$ does not depend on $v$. Hence together with  Lemma~\ref{lem:coarea},  we can now write
\[
{F_2}(v) = \int_0^1 \left( \int_{\R^2} |\chi_{E(t)_x}| + |\chi_{E(t)_y}| \,\, + \,\, \lambda \int_{\R^2} |\chi_{E(t)}-f|  \right) - \lambda C = \int_0^1 \overline{F_1}(\chi_{E(t)}) - \lambda C.
\]
Since $v$ is a minimizer of ${F_2}$, we find that for almost every $t\in [0,1]$, $\chi_{E(t)}$ is a minimizer of $\overline{F_1}$.   

\end{proof}

We now turn to minimizers of ${F_2}$. In order to 
 stay within the general framework of convex analysis, we have to define our functional on a vector space.  Thus we define for all $u\in BV(\R^2)$, the functional
\[
\widehat{F_2}(u) := {F_2}(u) + \zeta_{BV(\R^2; [0, 1])}(u),
\]
where for  a given set $A$,  the  $\zeta_A$ is defined as
\[
\zeta_A(x) := \left\{ \begin{array}{ll} 0 & \text{if } x\in A,\\ +\infty & \text{if } x\in A^c.\end{array}\right.
\]
Note that minimizing $\widehat{F_2}$ over $BV(\R^2)$ is equivalent to minimizing ${F_2}$ over $BV(\R^2; [0, 1])$.

We will now formulate results that tell us which conditions are necessary and/or sufficient for $u\in BV(\R^2)$ to be a minimizer of $\widehat{F_2}$. 
First we address the implications from convex analysis for minimizers of $\widehat{F_2}$. 
\begin{lemma}\label{lem:subdifs}
$u_0\in BV(\R^2)$ is a minimizer of $\widehat{F_2}$ over $BV(\R^2)$ if and only if $u_0\in BV(\R^2, [0, 1])$ and there exists $v\in L^\infty(\R^2; \R^2)$ satisfying
\begin{equation}\label{eq:measurecondtion}
\forall \mu \in \mathcal{M}_2(\R^2)\quad  {_{\mathcal{M}_2(\R^2)^*}}\langle v, \mu - \nabla u_0 \rangle_{\mathcal{M}_2(\R^2)} + \|\nabla u_0\|_a \leq \|\mu\|_a
\end{equation}
and in addition there exists $\xi\in BV(\R^2)^*$ such that
\[
\forall  u\in BV(\R^2; [0, 1]) \, {_{BV(\R^2)^*}}\langle \xi, u-u_0\rangle_{BV(\R^2)} \leq 0
\]
and
\[
\nabla^* v + \lambda \big( |1-f|-|f| \big) + \xi = 0.
\]

\end{lemma}
\begin{proof}
Following Lemma~\ref{lem:minimimpliesvectorf}, we consider the consequences of  $0 \in  \partial \widehat{F_2}(u_0)$, focusing on each of the three terms separately. The continuity of $F_2$ ensures that even though $\zeta_{BV(\R^2; [0, 1])}$ is not continuous with respect to the topology on $BV(\R^2)$ we can still use \cite[Chapter I, Proposition 5.6]{EkelandTemam76} to compute the subdifferential of each term separately and then add them to find $\partial \widehat{F_2}(u_0)$.

The subdifferential of the functional $BV(\R^2) \to \R: u \mapsto \int_{\R^2} |u_x| + |u_y|$ 
was analyzed  in Lemma~\ref{lem:minimimpliesvectorf} where we found that 
 $\chi \in \partial (\|\cdot\|_a \circ \nabla)(u_0)$ if and only if $\chi = \nabla^* v$ for a $v\in L^\infty(\R^2; \R^2)$ satisfying 
\[
\forall \mu \in \mathcal{M}_2(\R^2)\quad  {_{\mathcal{M}_2(\R^2)^*}}\langle v, \mu - \nabla u_0 \rangle_{\mathcal{M}_2(\R^2)} + \|\nabla u_0\|_a \leq \|\mu\|_a.
\]
Since the second term in $F_2$, i.e. the functional $BV(\R^2) \to \R: u \mapsto \int_{\R^2} \big(|1-f| -|f|\big) u$, is G{\^a}teaux differentiable, we find (cf. \cite[Chapter I, Proposition 5.3]{EkelandTemam76}) that its subdifferential is the singleton $\{(|1-f| - |f|)\}$. To be precise, the subdifferential at $u_0$ has exactly one element $\psi$ which satisfies, for all $u\in BV(\R^2)$,  
\begin{equation}\label{eq:psi=1-f-f}
{_{BV(\R^2)^*}}\langle \psi, u \rangle_{BV(\R^2)} = \int_{\R^2} \big(|1-f| - |f|\big) \, u.
\end{equation}

Turning to the last term in $\widehat{F_2}$, we have 
\begin{align*}
\xi \in \partial \zeta_{BV(\R^2; [0, 1])}(u_0) \Leftrightarrow &\Big[\zeta_{BV(\R^2; [0, 1])}(u_0) < \infty \text{ and } \forall u\in BV(\R^2) \,\\ &\hspace{0.3cm} {_{BV(\R^2)^*}}\langle \xi, u-u_0\rangle_{BV(\R^2)} + \zeta_{BV(\R^2; [0, 1])}(u_0) \leq \zeta_{BV(\R^2; [0, 1])}(u).\Big]
\end{align*}
The condition $\zeta_{BV(\R^2; [0, 1])}(u_0) < \infty$ is equivalent to $u_0\in BV(\R^2; [0, 1])$. For such $u_0$ we have $\zeta_{BV(\R^2; [0, 1])}(u_0) = 0$ and thus the second condition becomes
\[
\forall u\in BV(\R^2) \, {_{BV(\R^2)^*}}\langle \xi, u-u_0\rangle_{BV(\R^2)} \leq \zeta_{BV(\R^2; [0, 1])}(u).
\]
For $u\in BV(\R^2)\setminus BV(\R^2; [0, 1])$ this condition is trivially satisfied, while for $u\in BV(\R^2; [0, 1])$ we have $\zeta_{BV(\R^2; [0, 1])}(u)=0$. We deduce
\[
\xi \in \partial \zeta_{BV(\R^2; [0, 1])}(u_0) \Leftrightarrow u_0 \in BV(\R^2; [0, 1]) \text{ and } \forall  u\in BV(\R^2; [0, 1]) \, {_{BV(\R^2)^*}}\langle \xi, u-u_0\rangle_{BV(\R^2)} \leq 0.
\]

Adding the three computed subdifferentials gives the result.
\end{proof}

\begin{remark}
It is instructive to note that if $u \in BV(\R^2; [0, 1])$ is a minimizer of $\widehat{F_2}$ the condition
\[
\nabla^* v + \lambda \big( |1-f|-|f| \big) + \xi = 0
\]
from Lemma~\ref{lem:subdifs} does not allow us to conclude that $\nabla^* v =-\dvg v \in L^\infty(\R^2)$, as we could conclude in Theorem~\ref{thm:minimizervectorfield}.
\end{remark}

We now turn to a sufficient condition for $u$ to be a minimizer which can be adapted to deal with binary $u$ as well. 
Here more regularity on the vector field is  required as explained in the next lemma, which is based upon \cite[Proposition 3.3]{EsedogluOsher04} and \cite[Lemma 5.5]{ChanEsedoglu05}.

\begin{lemma}\label{lem:whenminimF2}
$u_0\in BV(\R^2)$ is a minimizer of $\widehat{F_2}$ over $BV(\R^2)$ if $u_0\in BV(\R^2; [0, 1])$ and in addition there exists a vector field $v\in \mathcal{V}(u_0)$ and for all $u\in BV(\R^2; [0, 1])$
\begin{equation}\label{eq:nonposparing}
\int_{\R^2} \Big(\text{div}\, v - \lambda \big(|1-f|-|f|\big) \Big) (u-u_0) \leq 0.
\end{equation}
If the inequality in (\ref{eq:nonposparing}) is strict, $u_0$ is the unique minimizer of $\widehat{F_2}$ over $BV(\R^2)$.

%

\end{lemma}
\begin{proof}
Because $u_0\in BV(\R^2; [0, 1])$ we have $\widehat{F_2}(u_0) = F_2(u_0) < \infty$. Therefore we have to show that $F_2(u_0) \leq F_2(u)$ for all $u\in BV(\R^2; [0, 1])$. By (\ref{eq:nonposparing}) and condition~\ref{item:condition3} in (\ref{eq:setofvfields}) we compute
\begin{align*}
F_2(u) &= \int_{\R^2} |u_x|+|u_y| + \lambda \int_{\R^2} (|1-f|-|f|) u\\
&= \int_{\R^2} |u_x|+|u_y| + \int_{\R^2} (u_0-u) \big(\text{div}\, v - \lambda (|1-f|-|f|)\big)\\ &\hspace{1cm} + \int_{\R^2} u_0 \big( \lambda (|1-f|-|f|) - \text{div}\, v \big) + \int_{\R^2} u\, \text{div}\, v\\
&\geq \int_{\R^2} |u_x|+|u_y| + \int_{\R^2} u\, \text{div}\, v + \int_{\R^2} u_0 \big( \lambda (|1-f|-|f|) - \text{div}\, v \big)\\
&= \int_{\R^2} |u_x|+|u_y| + \int_{\R^2} u\, \text{div}\, v + \int_{\R^2} |u_{0_x}| + |u_{0_y}| + \lambda \int_{\R^2} u_0\, (|1-f|-|f|)\\
&= \int_{\R^2} |u_x|+|u_y| + \int_{\R^2} u\, \text{div}\, v + F_2(u_0)\\
&\geq F_2(u_0).
\end{align*}
The last inequality follows since by Corollary~\ref{cor:anisotvLinfty} we have
\[
\int_{\R^2} |u_x|+|u_y| \geq -\int_{\R^2} u\, \text{div}\, v.
\]
If the inequality in (\ref{eq:nonposparing}) is strict, then so is the first inequality in the computation above, and hence $F_2(u_0)>F_2(u)$.
\end{proof}

\begin{corol}\label{F2-barcode}
Let $f\in \mathcal{B}$ be the measured signal in $\overline{F_1}$. Then $f$ is a minimizer of $\overline{F_1}$ over $BV(\R^2; \{0,1\})$ if there exists a vector field $v\in \mathcal{V}(f)$ such that
\begin{equation}\label{eq:samecondition4}
\lambda \geq \max\left( -\dvg v|_{\supp f}, \dvg v|_{(\supp f)^c}\right).
\end{equation}
If the inequality in (\ref{eq:samecondition4}) is strict, then $f$ is the unique minimizer of $\overline{F_1}$ over $BV(\R^2; \{0,1\})$.
\end{corol}
\begin{proof}
First we use Lemma~\ref{lem:whenminimF2} to prove that $f$ is a minimizer of $\widehat{F_2}$ over $BV(\R^2)$, and hence of ${F_2}$ over $BV(\R^2; [0,1])$. It then follows from Lemma~\ref{lem:convexify} that $f$ is also a minimizer of $\overline{F_1}$ over $BV(\R^2; \{0, 1\})$.
To this end, we show that the conditions of Lemma~\ref{lem:whenminimF2} are fulfilled with $u_0=f$. Most of them follow directly, only (\ref{eq:nonposparing}) in Lemma~\ref{lem:whenminimF2} needs some explanation.
Let $w\in BV(\R^2; [0,1])$. Since $f\in \mathcal{B}$ we have
\[
|1-f| -|f| = \left\{ \begin{array}{ll} -1 & \text{on } \supp f,\\ 1 & \text{on } (\supp f)^c, \end{array}\right. \quad \text{and} \quad w-f \left\{ \begin{array}{ll} \leq 0 & \text{on } \supp f,\\ \geq 0 & \text{on } (\supp f)^c. \end{array}\right.
\]
Because $\lambda \geq \max\left( -\dvg v|_{\supp f}, \dvg v|_{(\supp f)^c}\right)$ we find 
\begin{align*}
&\int_{\R^2} \Big(\text{div}\, v - \lambda \big(|1-f|-|f|\big) \Big) (w-f) =\\
& \hspace{2cm} -\int_{\supp f} \big(\text{div}\, v + \lambda \big) (f-w) + \int_{( \supp f)^c} \big(\text{div}\, v - \lambda \big) (w-f)
 \leq 0.
\end{align*}

Finally, if the inequality in (\ref{eq:samecondition4}) is strict it follows by the computation above that the inequality in (\ref{eq:nonposparing}) is strict as well and hence by Lemma~\ref{lem:whenminimF2} $f$ is the unique minimizer of $\widehat{F_2}$ over $BV(\R^2)$. Because $f$ is binary all its super level sets are the same and it follows by Lemma~\ref{lem:convexify} that $f$ is the unique minimizer of $\overline{F_1}$.

\end{proof}

Note that, as expected,  the condition on $\lambda$ we found for minimizing $F_1$, i.e. $\lambda \geq \|\dvg v\|_{L^\infty(\R^2)}$, implies condition (\ref{eq:samecondition4}). 
One could ask whether  the new condition on $\lambda$ is weaker in practice than the old one. This is only the case if $\dvg v$ takes either large positive values on $\supp f$ or large negative values on $(\supp f)^c$, which is not true for the two examples of vector fields $v$ we have seen in Theorem~\ref{thm:vectorfieldforbar code} and Remark~\ref{rem:differentvectorfield}. 

\section{Minimizers of $F_3$}\label{sec:minF3} 

\begin{lemma}\label{lem:F3mingivesvfield}
Let 
$K$ be as in (\ref{eq:whatisK}). If $u_0\in BV(\R^2)$ is a minimizer of $F_3$ over $BV(\R^2)$, then there exists a vector field $v\in \mathcal{V}(u_0)$ such that for all $w\in BV(\R^2)$,
\begin{equation}\label{eq:condwKw}
\left| \int_{\R^2} w\, \dvg v\right| \leq \lambda \int_{\R^2} |Kw|.
\end{equation}

\end{lemma}
\begin{proof}
The proof is almost identical to the proof of Lemma~\ref{lem:minimimpliesvectorf}. Here we point out the differences. The subdifferential of the anisotropic total variation is computed as before and hence the existence of a $v\in L^\infty(\R^2; \R^2)$ satisfying condition~\ref{item:condition1} in (\ref{eq:setofvfields}) follows as before.
For the subdifferential of the fidelity term at $u_0$, we find
\begin{align*}
\psi \in \partial \left(\int_{\R^2} |K\cdot - f|\right)(u_0) &\Leftrightarrow \int_{\R^2} |Ku_0 - f| < \infty \text{ and } \forall u\in BV(\R^2)\,\\
&\hspace{0.8cm} {_{BV(\R^2)^*}}\langle \psi, u-u_0 \rangle_{BV(\R^2)} \,\,  + \,\, \int_{\R^2} |Ku_0-f| \, \leq \, \int_{\R^2} |Ku-f|.
\end{align*}
Since $Ku_0\in L^1(\R^2)$, the condition $\displaystyle \int_{\R^2} |Ku_0 - f| < \infty$ is trivially satisfied.
Let $w\in BV(\R^2)$ and choose $u=\pm w+u_0$ in the second inequality in the right hand side above. Then we compute
\[
\pm {_{BV(\R^2)^*}}\langle \psi, w \rangle_{BV(\R^2)} \,\,  + \,\, \int_{\R^2} |K u_0-f| \leq \int_{\R^2} |\pm Kw + Ku_0 - f| \leq \int_{\R^2} |Kw| \,\, + \,\, \int_{\R^2} |Ku_0-f|,
\]
Hence by Lemma~\ref{lem:K}
\begin{equation}\label{eq:wKw}
\left|  {_{BV(\R^2)^*}}\langle \psi,w \rangle_{BV(\R^2)}  \right| \, \leq\,  \int_{\R^2} |Kw| \leq \int_{\R^2} |w|,
\end{equation}
for all $w \in BV(\R^2)$. 
The scaling arguments following equation (\ref{eq:psiulequ}) in the proof of Theorem~\ref{thm:minimizervectorfield}
now follow as before and we find that
as a distribution $\psi\in BV(\R^2)^*$ can be represented by $\psi \in L^\infty(\R^2)$, $\displaystyle -\dvg v + \lambda \psi = 0$ as distributions, and by density of $C_c^\infty(\R^2)$ in $L^1(\R^2)$
\[
\lambda\, \int_{\R^2} \psi \, w \, dx  = \int_{\R^2} w\, \dvg v
\]
for all $w\in L^1(\R^2)$. 
Combining this with the first inequality in (\ref{eq:wKw}) gives (\ref{eq:condwKw}).
From here on all the arguments follow as in the proof of Lemma~\ref{lem:minimimpliesvectorf}. 

\end{proof}

\begin{remark}\label{rem:wKWimplieslambda}
It is noteworthy that (\ref{eq:condwKw}) implies $\displaystyle \lambda \geq \|\dvg v\|_{L^\infty(\R^2)}$ as follows: By Lemma~\ref{lem:K} we have $\displaystyle \int_{\R^2} |Kw| \leq \int_{\R^2} |w|$ and hence (\ref{eq:condwKw}) allows us to repeat the argument in (\ref{eq:getlambdabound}).
\end{remark}

\begin{theorem}\label{thm:F3vectorfield}
Let $u_0\in BV(\R^2)$ and $K$ be as in (\ref{eq:whatisK}).
Define $f:=Ku_0$. Then $u_0$ is a minimizer of $F_3$ over $BV(\R^2)$ if and only if there exists a vector field $v\in \mathcal{V}(u_0)$ such that, for all $w\in BV(\R^2)$, (\ref{eq:condwKw}) holds.
%
%
Moreover, if such a vector field exist and the inequality in (\ref{eq:condwKw}) is strict, then $u_0$ is the unique minimizer of $F_3$ over $BV(\R^2)$.

\end{theorem}
\begin{proof}
By Lemma~\ref{lem:F3mingivesvfield}, it suffices to prove  that  if  the vector field  $v$ exists then $u_0$ is a minimizer of $F_3$ over $BV(\R^2)$.
Let $u\in BV(\R^2)$. By Corollary~\ref{cor:anisotvLinfty} we have
\[
\int_{\R^2} |u_x| + |u_y| \geq -\int_{\R^2} u\, \dvg v.
\]
Also note that by condition~\ref{item:condition3} in (\ref{eq:setofvfields}) we have $F_3(u_0) = -\int_{\R^2} u_0\, \dvg v$.
Using these results we find
\begin{align*}
F_3(u) &\geq -\int_{\R^2} u\, \dvg v + \lambda \int_{\R^2} |Ku-f| \\
&= -\int_{\R^2} u_0\, \dvg v + \lambda \int_{\R^2} |Ku-f| - \int_{\R^2} (u-u_0) \, \dvg v\\
&= F_3(u_0) + \lambda \int_{\R^2} |K(u-u_0)| - \int_{\R^2} (u-u_0)\, \dvg v \\
&\geq F_3(u_0).
\end{align*}
The inequality 
follows directly from (\ref{eq:condwKw}).
Finally, if the inequality in (\ref{eq:condwKw}) is strict
the above inequality is strict: $F_3(u) > F_3(u_0)$.

\end{proof}

The following lemma shows that if the linear operator $K$ is given by convolution with a blurring kernel/PSF,  condition (\ref{eq:condwKw}) is satisfied for  $\lambda$ sufficiently  large.

\begin{lemma}\label{lem:Kisconvolution}
Let $u_0\in BV(\R^2)$ and $v\in \mathcal{V}(u_0)$. For all $w\in BV(\R^2)$ define $Kw:=k*w$ for some nonnegative kernel $k\in L^1(\R^2)$ satisfying $\int_{\R^2} k = 1$. If $\lambda \geq \|\dvg v\|_{L^\infty(\R^2)}$ condition (\ref{eq:condwKw}) holds. Moreover, if $\lambda > \|\dvg v\|_{L^\infty(\R^2)}$ the inequality in (\ref{eq:condwKw}) is strict.
\end{lemma}
\begin{proof}
Because 
\[
\int_{\R^2}  \left|k \ast (w\, \dvg v) \right| \le  \|\dvg v\|_{L^\infty(\R^2)} \int_{\R^2}  |k \ast w|
\]
the convolution $k \ast (w\, \dvg v)$ is well-defined. By Lemma~\ref{lem:K} 
we conclude that
\[ \left| \int_{\R^2} w\, \dvg v\right|  \le   \int_{\R^2}  \left|k \ast (w\, \dvg v) \right| \le   \|\dvg v\|_{L^\infty(\R^2)}  \int_{\R^2} |k*w| \leq \lambda  \int_{\R^2} |Kw|.
\]
The last inequality is strict if $\lambda > \|\dvg v\|_{L^\infty(\R^2)}$.
\end{proof}

When we compare Theorem~\ref{thm:F3vectorfield} to Theorem~\ref{thm:minimizervectorfield},  we see that there are two differences in the conditions on the vector field $v$: (i) For Theorem~\ref{thm:F3vectorfield} condition~\ref{item:condition3} in (\ref{eq:setofvfields}) involves $z$ instead of $f$, as it did for Theorem~\ref{thm:minimizervectorfield}; (ii) The combined condition (\ref{eq:condwKw}) on $\lambda$ and $K$ from Theorem~\ref{thm:F3vectorfield} is stronger than the condition on $\lambda$ we had before in Theorem~\ref{thm:minimizervectorfield}. Hence with Lemma~\ref{lem:Kisconvolution} in mind, we can transfer all the results in Section~\ref{sec:Q2} which we derived for $F_1$ with $f\in \mathcal{B}_\omega$ to $F_3$ with $f=k*z$ for $z\in \mathcal{B}_\omega$. In particular, we have 
\begin{theorem}\label{barcode-conv}
Let $z\in \mathcal{B}_\omega$, let $Ku := k*u$ for $u\in BV(\R^2)$ and a nonnegative $k\in L^1(\R^2)$ satisfying $\int_{\R^2} k = 1$, and let $f=k*z$. If $\lambda \ge \frac4{\omega}$, 
then $z$ is a minimizer of $F_3$ over $BV(\R^2)$. If  the inequality  is strict, then $z$ is the unique minimizer of $F_3$ over $BV(\R^2)$.
\end{theorem}

\begin{erratum}\label{err:F3onlyfaithful}
In versions of this paper prior to arXiv:1007.1035v3, Theorem 6.5 contained a second part, which stated that: if $u_0\in BV(\R^2; \{0,1\})$, $K$ as in (\ref{eq:whatisK}), $f=Ku_0$, and $u_0$ is a minimizer of $F_3$ over $BV(\R^2)$, then $u_0\in \mathcal{B}$. This statement was based on an incorrect lemma that was included of Section~\ref{sec:Q2} and has now been removed, as explained in Erratum~\ref{err:counterexample}. Hence this second part of the theorem has now also been removed from the paper.
\end{erratum}

\begin{remark} It is  interesting to note in Theorem~\ref{barcode-conv} that the conditions on $\lambda$ for recovery of the bar code do not depend on properties of the blurring/deblurring kernel $k$. In \cite{ChoksiGennip09}, we considered the problem of deblurring of 1D bar codes. In 1D there is no difference between our anisotropic and the regular isotropic total variation, however we did employ an $L^2$ instead of $L^1$ fidelity term. While our main focus was for  deblurring and blurring kernels of different size, a corollary of our results was that when the two coincided (analogous to $F_3$), the functional was faithful to 
the clean bar code  for  blurring kernels with modest supports (on the order of the $X$-dimension). Numerical results suggested that this bound on the size of the blurring kernel was not optimal. 
Our Theorem \ref{barcode-conv} can readily be adopted to 1D barcodes, showing that regardless of the support size of the kernel  (or the standard deviation of an infinitely supported kernel), deconvolution with the same blurring kernel always recovers the barcode for $\lambda \geq \frac2\omega$ (via the vector field construction of Theorem~\ref{thm:vectorfieldforbar code} reduced to 1D) when using an $L^1$ fidelity term. However, as we note in Section 
\ref{sec:numerical}, this threshold value for $\lambda$ is very sensitive to noise and indeed this sensitivity grows with the support size of the blurring kernel. 
\end{remark} 

\section{Numerical Implementation}\label{Numerical-Implementation}
We have numerically tested the performances of the convex functionals  $F_i$.   As convex functionals,    their minimization can be approximated as finite dimensional convex optimization problems and global minimizers of these problems can be found using standard software packages.
We chose such an implementation because it allows us to find global minimizers without writing custom algorithms for each functional.   This allows for 
convenient comparison and experimentation with different functionals.   
In some cases, other options are available.  For example, for the functional $F_1$ gradient descent methods on a regularized functional can be used as in~\cite{ChanEsedoglu05},~\cite{ChanEsedogluNikolova06}.   We are not aware of direct methods for~$F_3$. 

The discretization is obtained by using standard forward finite differences and quadrature. 
Because of the particular form of the anisotropic total variation, each of the problems can be reformulated as a linear program, which is usually more tractable than a general convex optimization problem.   
By comparison, the standard isotropic total variation involves a term of the form $\sqrt{u_x^2 +u_y^2}$ which does not allow for a linear program reformulation, and leads to a more challenging optimization problem. 

Let us show how to discretize and reformulate $F_3$ as a linear program --  
the functionals $F_1$ and $F_2$ can be discretized and minimized in a similar manner (in fact $F_1$ is a special case of $F_3$, where the convolution matrix is the identity).
For a given small parameter $h$, approximate the minimizer $u(x,y)$ by a suitable (piecewise linear or piecewise constant) interpolation $u^h(x,y)$ of a grid function $U$, defined on the grid with spacing $h$, where $i,j = 1,\dots, N$.   
Here $u(ih, jh) = U_{i,j}$.
Next, approximate the terms $u_x$ and $u_y$ by finite differences
\begin{align*}
u_x(ih, jh) &= \frac{1}{h} \left ( U_{i+1, j} - U_{i,j} \right)\\
u_y(ih, jh) &= \frac{1}{h} \left ( U_{i, j+1} - U_{i,j} \right).
\end{align*}
For the purpose of building matrices for the linear operators, reindex $U$ as a column vector of length $N^2$.
Write $D_x, D_y$ for ($h$ times) the matrices which correspond to the finite difference operators above.   
Similarly, the convolution operator can be approximated by a convolution with a discrete kernel $K^h$.  Let $M_k$ represent the matrix of the convolution operator, also scaled by the factor $h$.   
{The matrices  $D_x, D_y, M_k$ each have $N^2$ columns.  They have different numbers of rows.  Denote the number of rows for each matrix by  {$m_1, m_2, m_3$}, respectively.
Let $F$ be the grid function which corresponds to a blurred and noisy vector and   represent $F$ as a column vector of length $m_3$.}
For simplicity, we use piecewise constant quadrature to approximate the integrals in the operator $F_3$, and, after dropping a factor of $h$, we are left with 
\[
F^h_3(U) = \| D_x U \|_1 + \| D_y U\|_1 + {\lambda}{h} \| M_k U - F \|_1,
\]
 a fully discrete convex function of the grid function $U$.

To formulate the equivalent   linear program, define the matrix and vector
\[
M =
\left[\begin{array}{c}D_x \\D_y \\ {M_k} \end{array}\right], \quad 
b = \left[\begin{array}{c}0_x \\0_x \\F\end{array}\right].
\]
Here $0_x, 0_y$ are column vectors of zeros of length $m_1, m_2$, respectively.
Using this definition, we can rewrite
\[
F^h_3(U) = \| M U - b \|_1.
\]
Next, changing notation slightly, (the following notation applies just to this section) we show how to recast the problem
\[
\min_{x \in \R^{N^2}} \| Mx - b\|_1
\] 
as a linear programming problem.   
Define the new variable $y^t = ({x^+}^t, {x^-}^t, x^t)$, where $x^+, x^-$ are column vectors in $\R^{N^2}$, and the vector $e^t = (1,\dots, 1)\in \R^{N^2}$.  
Then consider the equivalent problem
\[
\begin{aligned}
    \minimize{y \in \R^{3N^2}}~ &{e^t x^+ + e^t x^- }
  \\ \text{ subject to }  & 
  \begin{cases}    
  	  Mx - b = x^+ - x^- 
  		\\ 
		x^+, x^- \ge 0.
  \end{cases}
  \end{aligned}
\]
obtained by splitting $Mx-b$ into a positive and negative part and summing their 1-norms. This is a linear program: It involves the minimization of the linear function $(c^t y)$, for $c = (e^t, e^t, 0)$, subject to a linear equation, and non-negativity constraints on a subset of the variables.  

We performed the minimization using two convex optimization packages, \texttt{CVX}, and \texttt{MOSEK}.  Both are callable from {MATLAB}.  
{The first, \texttt{CVX}~\cite{CVX,gb08}, is a package for specifying and solving general convex optimization problems.
The second, \texttt{MOSEK}~\cite{mosek}, requires the user to reformulate each problem as a standard form optimization problem (in this case as a linear program), but it is able to solve larger problems.}
For the problems we presented,~\texttt{CVX} was able to solve the problem in a few seconds for $F_1, F_2$ and a few minutes for smaller instances of $F_3$, \texttt{MOSEK} was able to solve the small instances of $F_3$ in seconds, and the largest problems in a few minutes.

\section{Numerical Results}\label{sec:numerical}
We compared the performance of the functionals  $F_i$ applied to a number of 2D bar codes corrupted by some combination of noise and blurring.   To produce blurred and noisy images, convolution with blurring kernels of variable sizes  was followed by additive Gaussian noise. 
The objective was to test the conditions for which bar codes could be recovered and to compare the performance of the various functionals, not to optimize the numerical implementation. 
The bar code images had a resolution of either $8 \times 8$  or $12\times 12$ pixels per square of size $\omega\times\omega$.  That is, we choose $h = \omega / p$ where $p$ is the number of pixels per square. In all cases, we give values for  the dimensionless fidelity parameter  $\overline{\lambda}: = \lambda h$.

\subsection{Noisy Images}
We begin with a comparison of performance of the two functionals, $F_1$ and ${F_2}$ on noisy images.  
We show the same noisy bar code denoised with the two methods (\autoref{fig:noisy1}).   
In each run we used $\overline{\lambda} = 75$ for $F_1$ and $\overline{\lambda}= 2$ for $F_2$.
In contrast to $F_1$, minimization of  ${F_2}$ resulted in a binary image 
 (even without taking a level set\footnote{This is no surprise if we expect the minimizers of $\overline{F_1}$ to be unique. It can also be proven directly for minimizers of $F_2$ taking discrete values on a square grid.}).  
We found that the functional $F_2$ was robust under fairly large noise and indeed, 
the reconstruction is nearly perfect even in the presence of significant amounts of Gaussian noise.   The performance of $F_2$ is superior to $F_1$ as it retains both the shape and the binary character of the image. We note that simply thresholding the result of $F_1$ can introduce additional errors.
\begin{figure}
  \centering
  \subfloat[Noisy Image (a=.2)]{
  \hspace{-.5in}
  \includegraphics[width=0.3\textwidth]{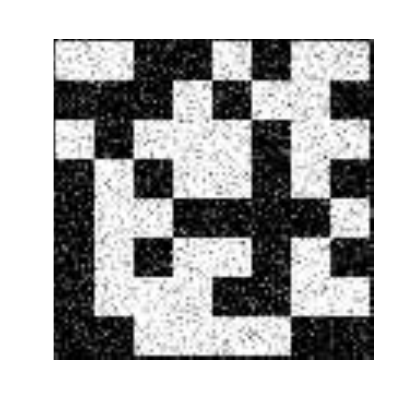}}                
  \subfloat[Result of $F_1$]{
  \includegraphics[width=0.3\textwidth]{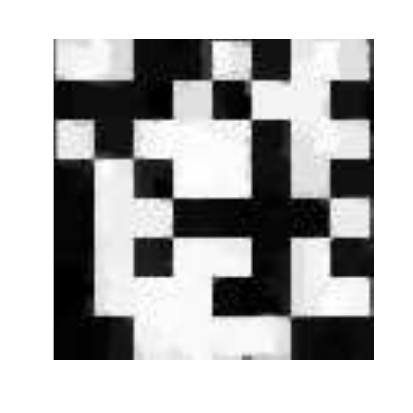}}
  \subfloat[Result of $F_2$]{
  \includegraphics[width=0.3\textwidth]{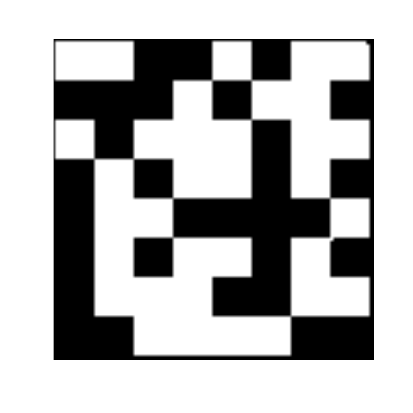}}
 
   \subfloat[Noisy Image (a=.35)]{
    \hspace{-.5in}   
  \includegraphics[width=0.3\textwidth]{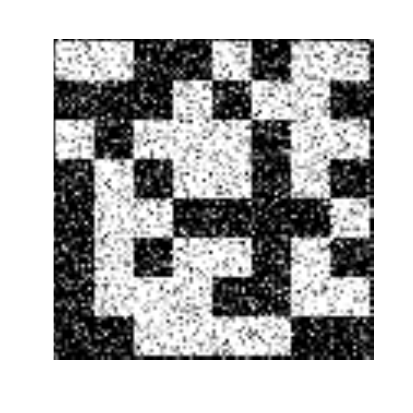}}                
  \subfloat[Result of $F_1$]{
  \includegraphics[width=0.3\textwidth]{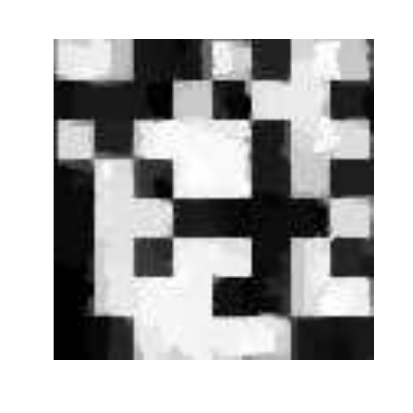}}
  \subfloat[Result of $F_2$]{
  \includegraphics[width=0.3\textwidth]{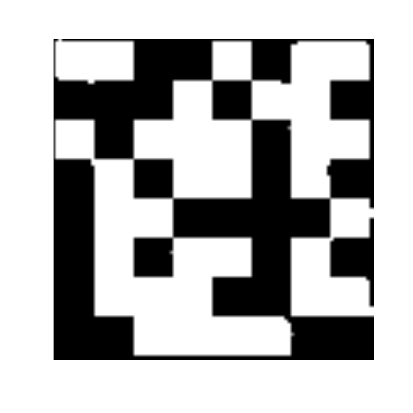}}

  \caption{Comparison of denoising using $F_1$ and $F_2$.  
 First row: Gaussian noise (amplitude  $a = .2$, signal to noise ratio 18.2 dB).  
 The reconstruction using $F_1$ is patchy. The reconstruction with $F_2$ is nearly perfect.
  Second row: Gaussian noise (amplitude $a=.35$, signal to noise ratio 7.0 dB).  
  The reconstruction using $F_1$ has deteriorated further.
  The reconstruction with $F_2$ is correct, except for a few switched pixels.
  }
\label{fig:noisy1}
\end{figure}

\subsection{Blurred Images}
The blurring kernel was chosen to be a piecewise linear hat function on a square.  
(In one dimension, for a given kernel radius $r$, the kernel is the normalization of the pixel function with values $(1,2,\dots r, \dots 2, 1)$.   The two-dimensional kernel is a Cartesian product of that one-dimensional kernel with itself.)
In each run, we used $\overline{\lambda} = 10$ for $F_3$.
The first result (\autoref{fig:blur1}) is for the bar code with a single square, blurred with a kernel whose radius is equal to the width of the square and with Gaussian noise added.  The reconstruction has close to the correct shape, but has lost some contrast.
The second result involves blurring without noise (Figure~\ref{fig:blurNoNoise}). 
The third result (\autoref{fig:blur2}) involves two different kernels.  
After thresholding, the recovered image is quite close to the original even in the presence of substantial noise. 
For the wider kernel there is some deterioration along diagonal patterns of squares.

\begin{figure}
  \centering
  \subfloat[Original]{
  \hspace{-.5in}
  \includegraphics[width=0.3\textwidth]{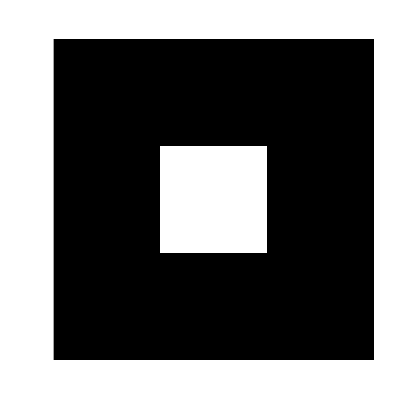}}                
  \subfloat[Blurred and Noisy]{
  \includegraphics[width=0.3\textwidth]{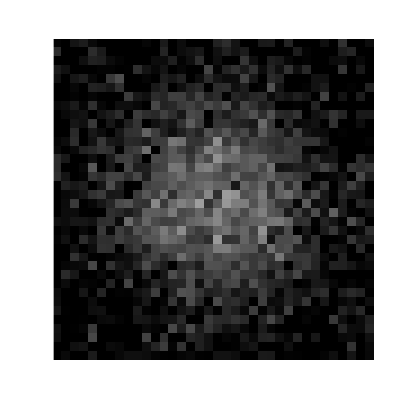}}
  \subfloat[result of $F_3$]{
  \includegraphics[width=0.3\textwidth]{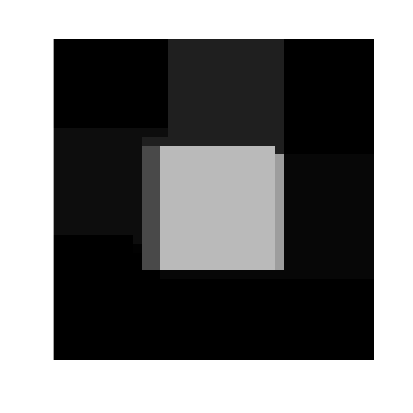}}
  \caption{Result of $F_3$ on a trivial bar code.  Each square is $12\times 12$ pixels, the blurring kernel radius is 12 pixels. Gaussian noise (amplitude .1, signal to noise ratio 23.0 dB).  The reconstruction loses some contrast.
}
\label{fig:blur1}
\end{figure}

\begin{figure}
  \centering
  \subfloat[Original]{
  \includegraphics[width=0.3\textwidth]{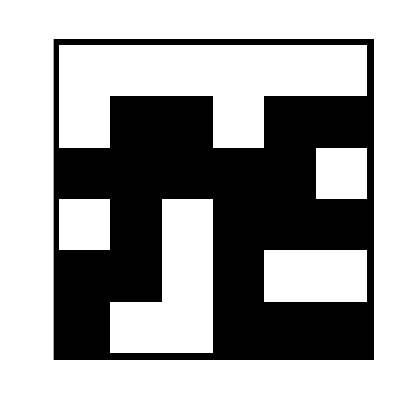}}                
  \subfloat[Blurred]{
  \includegraphics[width=0.3\textwidth]{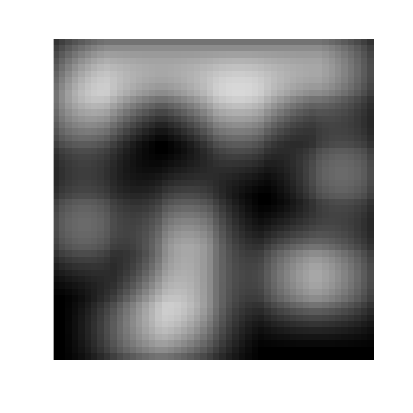}}
  
  \subfloat[$\overline{\lambda}=  1$]{
  \includegraphics[width=0.3\textwidth]{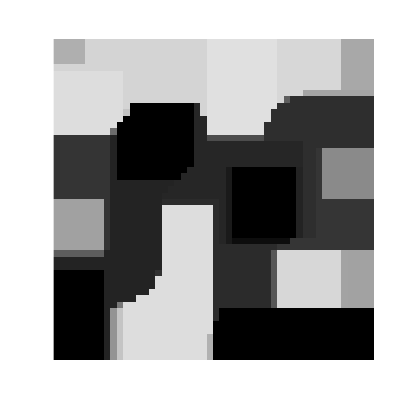}}
    \subfloat[$\overline{\lambda} = 4$]{
  \includegraphics[width=0.3\textwidth]{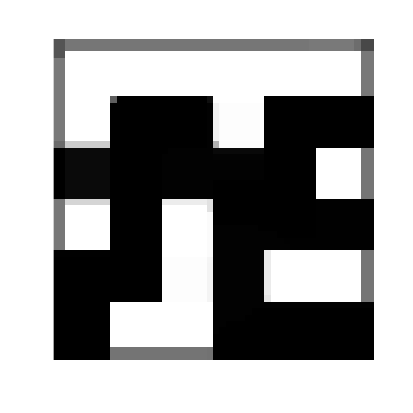}}
  \subfloat[$\overline{\lambda} = 8$]{
  \includegraphics[width=0.3\textwidth]{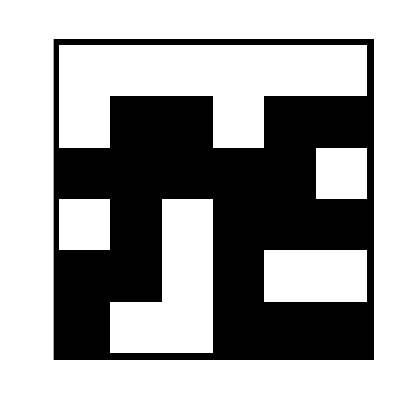}}

  \caption{Result of $F_3$ with a range of values of $\overline{\lambda}$ and no noise.  Each $\omega\times\omega$ square is $8\times 8$ pixels, the blurring kernel radius is 8 pixels.   }
\label{fig:blurNoNoise}
\end{figure}

\begin{figure}
  \subfloat[Original]{\hspace{-.5in}
  \includegraphics[width=0.25\textwidth]{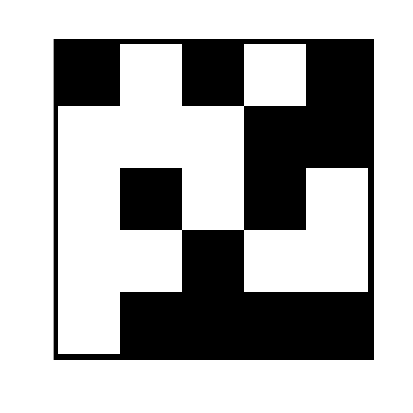}}                
  \subfloat[Blur r=8, Noise a =.02]{
  \includegraphics[width=0.25\textwidth]{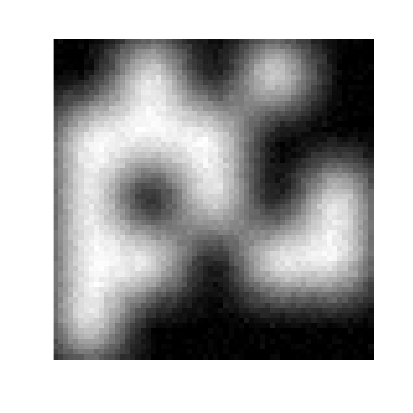}}
    \subfloat[result of $F_3$]{
   \includegraphics[width=0.25\textwidth]{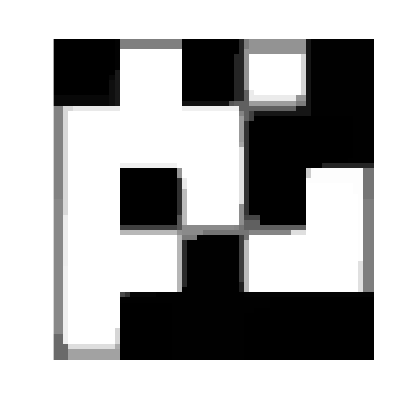}}
    \subfloat[result of $F_3$, thresholded]{
   \includegraphics[width=0.25\textwidth]{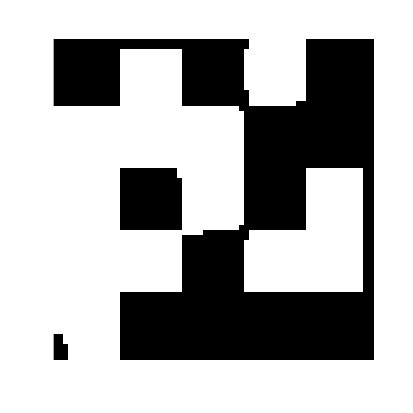}}


  \subfloat[Original]{\hspace{-.5in}
  \includegraphics[width=0.25\textwidth]{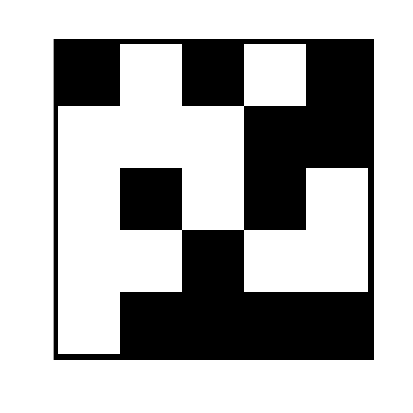}}                
  \subfloat[Blur r=8, Noise a=.2]{
  \includegraphics[width=0.25\textwidth]{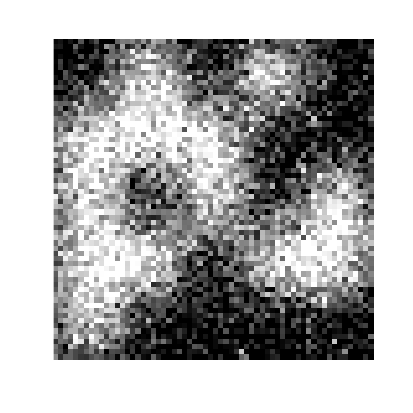}}
    \subfloat[result of $F_3$]{
   \includegraphics[width=0.25\textwidth]{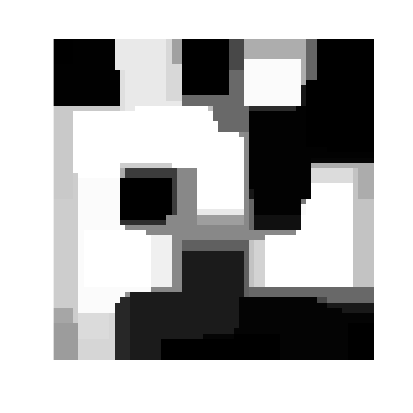}}
    \subfloat[result of $F_3$, thresholded]{
   \includegraphics[width=0.25\textwidth]{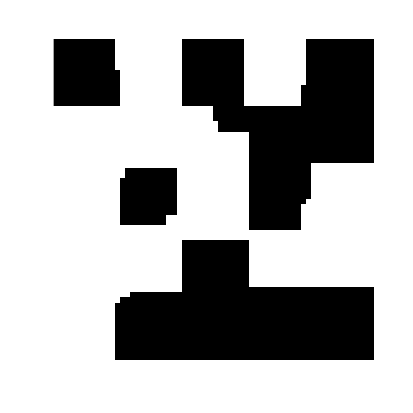}}

  \subfloat[Original]{\hspace{-.5in}
  \includegraphics[width=0.25\textwidth]{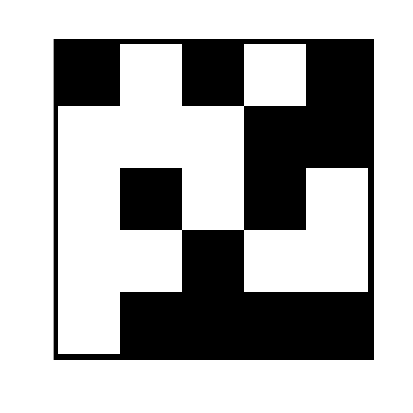}}                
  \subfloat[Blur r=12, Noise a=.2]{
  \includegraphics[width=0.25\textwidth]{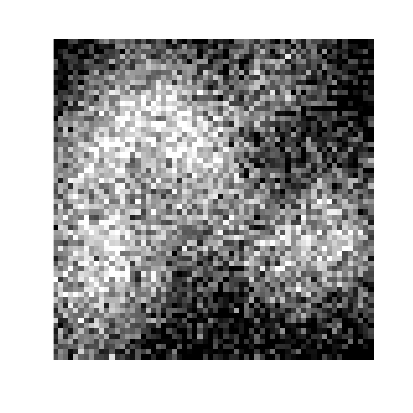}}
    \subfloat[result of $F_3$]{
   \includegraphics[width=0.25\textwidth]{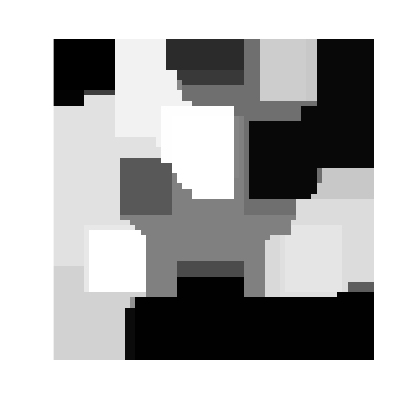}}
    \subfloat[result of $F_3$, thresholded]{
   \includegraphics[width=0.25\textwidth]{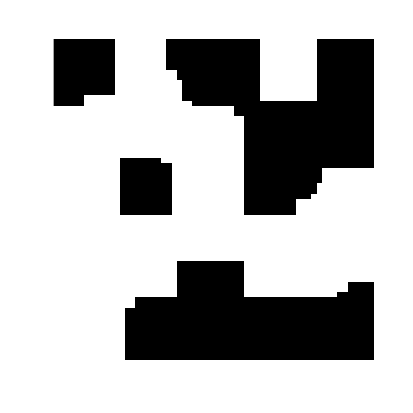}}
   \caption{Results of $F_3$. Each square in the bar code is $12\times 12$ pixels.  Gaussian noise amplitude $a$ is either .02 or .2. The blurring kernel radius $r$ is 8 or 12 pixels.  For the mildest example, the reconstruction is almost perfect.  Errors observed include deterioration along diagonal patterns of squares or incorrect placement of the boundaries of squares.
 However, note that even the last case is not that bad given that  the blurring kernel with $r=12$ has support spread over two unit squares.   }
\label{fig:blur2}
\end{figure}

One might ask how our values for $\overline\lambda$ compare with the results of our Theorems which basically state that if 
$\lambda \omega >4$, then the global minimizer is the underlying bar code (i.e. if $ \overline\lambda  > 4 / N$, where $N$ is the number of 
pixels per unit square). Comparison here is delicate because of the presence of noise. First off, there is an upper limit to acceptable values of $\overline\lambda$ as  large values will bring in too much unwanted fidelity to the noise. 
More importantly, in all cases we found the lower threshold value for $\overline\lambda$ to be very sensitive to noise:  for a given barcode signal,  even adding noise with amplitude $10^{-6}$ increased the threshold value.
Thus even in the simulations underlying Figure~\ref{fig:blurNoNoise} which involve no externally added noise, numerical round-off errors can account for enough noise to change the threshold for $\overline\lambda$. This explains the fact that the  threshold required for exact recovery was  larger than predicted by Theorem~\ref{barcode-conv}  (by a factor of 8).
A simulation with a narrow kernel, on a smaller image resulted in a critical value of $\overline\lambda$ which was off by (the smaller) factor of 1.4.  We expect that with wider kernels and larger problems, this value can be even more sensitive to noise.

In conclusion, in the presence of a known blurring kernel the functional $F_3$ followed by thresholding can recover close to the exact bar code for blurring kernels  with diameters about 1.5 times the $X$-dimension and small noise.   
Increasing the radius of the kernel or the amplitude of the noise leads to errors in the recovered image.   The types of errors included incorrect locations of boundaries of squares, or additional features along squares arranged in diagonal patterns. 
Not surprisingly, computations of blurred images (not presented) using $F_1$ or $F_2$ were less effective at recovering the bar code.

\section{Discussion}

We have presented three functionals for the $L^1$ approximation of signals with anisotropic total variation regularization and applied them to the denoising and deblurring of bar codes. Our analytical results show that the fidelity parameter $\lambda$ should be chosen above a certain threshold in order to not get the trivial minimizer. When comparing $F_1$ and $\overline{F_1}$ the analytical results in the absence of noise or blurring show that there is only a slight relaxation of the threshold for $\lambda$ to get the bar code back.  The numerical experiments clearly show that in practical situations with noise and blurring, $\overline{F_1}$ (or actually $F_2$) is preferred since it has binary output. 
If the blurring operator $K$ is known, it is even more advantageous to choose $F_3$. Our analysis shows that in the absence of noise, but with blurring present, we can recover the clean bar code by using $F_3$ regardless of the precise form of (the known) $K$. The numerical experiments also show the best results for $F_3$. 

There are several avenues for future work: 
\begin{itemize}

\item The convexification method we used for $\overline{F_1}$ is not applicable to the restriction of $F_3$ to $BV(\R^2; \{0, 1\})$. It would be valuable to see if there is another way to incorporate the binary restraint into $F_3$ apart from the \emph{a posteriori} thresholding as applied in the numerical experiments.

\item 
{\bf Blind deconvolution:} In practice, the blurring kernel underlying the measured signal $f$ may be unknown. 
Future analytical work could focus on  how well $F_3$ performs if the exact blurring operator is not known, and hence $K$ is only an approximation of the operator hidden in the measured signal $f$.  One possibility could be  to first try to determine certain statistics of the unknown convolution kernel (such that its standard deviation), and then  use $F_3$ with a $K$ consisting of convolution with a fixed kernel possessing the same statistics. 
Within a Gaussian ansatz, blind deconvolution was addressed variationally for 1D bar codes  in \cite{Esedoglu04} and a similar approach could also be adopted for 2D bar codes.

%

\item
The numerical experiments suggest that, for certain choices of $\lambda$, minimization of these functionals works well with both significant noise and blurring. In the case of non-noisy bar codes our theorems give sufficient thresholds for acceptable values of $\lambda$ but  the numerics show that these are sensitive to noise. 
It  would therefore be interesting to explore whether or not one can analyze the dependence of these thresholds on small perturbations of the signal $f$.

\item
{\bf Nonlocal total variation:}  A possible alternative regularization term instead of anisotropic total variation is anisotropic nonlocal total variation (cf. \cite{GilboaOsher07, GilboaOsher08})
\[
\int_{\R^2} \int_{\R^2} |u(x)-u(y)| \sqrt{w(x,y)}\,dx\,dy 
\]
for some well chosen weight function $w$. Nonlocal total variation does not restrict itself to local information, but compares patches from all over the image and hence is well suited to regularize images containing recurring structures, like bar codes. It would be interesting to see what analysis and simulations can tell us about the improvement this would be over the local anisotropic total variation.

\end{itemize}

\appendix

\section{Properties of the anisotropic total variation}\label{sec:anisoprops}

In this appendix we collect some properties  of our anisotropic total variation~\pref{eq:anisotv}, most of which follow from the analogous properties of the standard isotropic total variation which defines the space $BV(\R^2)$: $u \in L^1 (\R^2)$ is in the space $ BV (\R^2)$ iff 
\begin{equation}\label{TV}
\int_{\R^2} |\nabla u|  <  \infty \,\, {\rm where} \,\,\int_{\R^2} |\nabla u|  :=  \sup \left\{ \int_{\R^2} u \,\text{div} v : v\in C_c^1(\R^2; \R^2), \forall x, \, \|v \|_\infty \leq 1\right\}.
\end{equation}

We start by pointing out that the anisotropic total variation is an equivalent seminorm to the isotropic total variation defined above.
\begin{lemma}\label{lem:equivseminorm}
For $u\in BV(\R^2)$ we have
\[
\int_{\R^2} |\nabla u| \leq \int_{\R^2} |u_x| + |u_y| \leq \sqrt2 \int_{\R^2} |\nabla u|.
\]
\end{lemma}
Lemma~\ref{lem:equivseminorm}  immediately gives the following  isoperimetric inequality (cf. \cite[5.6.1 Theorem 1 (i)]{EvansGariepy92}). 
\begin{lemma}\label{lem:isoperimetric}
There exists a constant $C>0$ such that for all $u\in BV(\R^2)$
\[
C \|u\|_{L^2(\R^2)} \leq \int_{\R^2} |u_x| + |u_y|.
\]
\end{lemma}

The following approximation lemma allows us to replace $C_c^1$ functions with $L^\infty$ in the definition of our anisotropic total variation. 
\begin{lemma}\label{lem:approxv}
Let $v\in L^\infty(\R^2; \R^2)$ with $\dvg v \in L^\infty(\R^2)$ and $|v(x)|_\infty \leq 1$, then there exists a sequence $\{v_j\}_{j=1}^{\infty} \subset C_c^\infty(\R^2; \R^2)$ such that as $j\to\infty$, $v_j \overset{*}\rightharpoonup v$ in $L^\infty(\R^2; \R^2)$, and $\dvg v_j  \overset{*}\rightharpoonup \dvg v$ in $L^\infty(\R^2)$. That is,  for all $u\in L^1(\R^2; \R^2)$ and $w \in L^\infty(\R^2)$, as $j\to\infty$,
\[
\int_{\R^2} u\cdot v_j \to \int_{\R^2}u\cdot v \quad \text{and} \quad \int_{\R^2} w\, \dvg v_j \to \int_{\R^2} w\, \dvg v.
\]
In addition, for each $j$ and all $x\in\R^2$, $|v_j(x)|_\infty \leq 1$.
\end{lemma}
\begin{proof}
The proof is very similar to  \cite[Proposition 3.3]{EsedogluOsher04}. Let $\xi \in C_c^\infty(\R^2)$ be a cutoff function satisfying $\xi \leq 1$, $\xi(x)=1$ if $|x|<1$ and $\xi(x)=0$ if $|x|>2$ and let $\eta\in C_c^\infty(\R^2)$ be the standard radially symmetric mollifier with $\eta \geq 0$ and $\int_{\R^2} \eta = 1$. Define $v_j\in C_c^\infty(\R^2; \R^2)$ via
\[
v_j(x) := j^2 \big( \xi (x/j) v(x)\big) * \eta(jx)
\]
where the convolution is componentwise. One can readily check that this new sequence satisfies all the desired bounds on $\|v_j\|_{L^\infty(\R^2; \R^2)}$, $\|\dvg v_j\|_{L^\infty(\R^2)}$, and $|v(x)|_\infty$.   


\end{proof}
\noindent An immediate corollary to Lemma \ref{lem:approxv} is 
\begin{corol}\label{cor:anisotvLinfty}
\begin{equation}\label{eq:anisotvLinfty}
\int_{\R^2} |u_x| + |u_y| = \sup \left\{ \int_{\R^2} u \,\text{div} \,  v : v\in L^\infty(\R^2; \R^2), \dvg v \in L^\infty(\R^2),  \, |v(x)|_\infty\leq 1 \, {\rm a.e.} \right\}
\end{equation}
\end{corol}

Four important properties of the standard isotropic total variation also hold for the anisotropic total variation: lower semicontinuity, 
approximation by smooth functions, the co-area formula, and smooth approximation to sets of finite perimeter.   The proofs follow those of the isotropic case  (cf. \cite[Theorems 1.9, 1.17]{Giusti84}, \cite[5.5 Theorem 1]{EvansGariepy92}, \cite[Theorem 3.42]{AmbrosioFuscoPallara00}) with the obvious modifications. 

\begin{lemma}\label{lem:TVlsc}
For every sequence $\{u_n\}_{n=1}^\infty \subset BV(\R^2)$ such that $u_n \to u$ in $L_{\text{loc}}^1(\R^2)$ as $n\to \infty$, for some $u\in BV(\R^2)$, we have
\[
\int_{\R^2} |u_x| + |u_y| \leq \underset{n\to \infty}{\lim\inf}\, \int_{\R^2} |u_{n_x}| + |u_{n_y}|.
\]
\end{lemma}

\begin{lemma}\label{lem:approximateBV}
Let $u \in BV(\R^2)$, then there exists a sequence $\{u_n\}_{n=1}^\infty \subset C^\infty(\R^2)$ such that $u_n \to u$ in $L^1(\R^2)$ if $n\to\infty$ and
\[
\underset{n\to\infty}\lim\, \int_{\R^2} |u_{n_x}| + |u_{n_y}| = \int_{\R^2} |u_x| + |u_y|.
\]
\end{lemma}

\begin{lemma}\label{lem:coarea}
Let $f \in BV(\R^2)$ and define
\[
E(t) := \{x\in\R^2: f(x)>t\}.
\]
Then
\[
\int_{\R^2} |f_x| + |f_y| = \int_{-\infty}^\infty \left(\int_{\R^2} |\chi_{{E(t)}_x}| + |\chi_{{E(t)}_y}| \, \right) \, dt
\]
\end{lemma}

\begin{lemma}\label{lem:smoothapprox}
If $E$ is a set of finite perimeter in $\R^2$, then there exists a sequence $\{E_j\}_{j=1}^\infty$ of open sets with smooth boundaries converging in measure to $E$ and such that
\begin{align*}
\underset{j\to\infty}\lim\, \int_{\R^2} |\nabla \chi_{E_j}| = \int_{\R^2} |\nabla \chi_E| \quad \text{and}\\
\underset{j\to\infty}\lim\, \int_{\R^2} |\chi_{{E_j}_x}| + |\chi_{{E_j}_y}| = \int_{\R^2} |\chi_{E_x}| + |\chi_{E_y}|.
\end{align*}
\end{lemma}

\bigskip

\bigskip

{\bf Acknowledgments:} We thank Selim Esedo\=glu for many useful discussions  and suggestions. We are very grateful to Matthias R\"oger and Nils Dabrock for pointing out a mistake that was present in earlier versions of this paper (see Errata~\ref{err:onlyfaithful},~\ref{err:counterexample},~\ref{err:F3onlyfaithful}).
The research was supported by 
 NSERC (Canada) Discovery Grants. YvG was also supported by a PIMS postdoctoral fellowship at Simon Fraser University.

\bibliographystyle{acm}
\bibliography{bibliography}

\bigskip

\textbf{Electronic mail addresses of the authors:}
\begin{itemize}
\item Rustum Choksi: rchoksi@math.mcgill.ca
\item Yves van Gennip: yvgennip@math.ucla.edu
\item Adam Oberman: aoberman@math.sfu.ca
\end{itemize}

\end{document}